\documentclass{amsproc}

\usepackage{amsmath}%
\usepackage{amsthm}
\usepackage{thmtools}
\usepackage{amsfonts}%
\usepackage{amscd}
\usepackage{amssymb}%
\usepackage{graphicx}
\usepackage[all]{xy}
\usepackage{varioref}
\usepackage[hypertexnames=false]{hyperref}
\usepackage{mathtools}
\usepackage{stmaryrd}
\usepackage{xcolor}
\usepackage[pageref]{backref}
\usepackage{dsfont}
\usepackage{url}

\newtheorem{theorem}{Theorem}[section]
\newtheorem{lemma}[theorem]{Lemma}
\newtheorem{corollary}[theorem]{Corollary}
\newtheorem{proposition}[theorem]{Proposition}

\theoremstyle{definition}
\newtheorem{definition}[theorem]{Definition}
\newtheorem{example}[theorem]{Example}
\newtheorem{examples}[theorem]{Examples}

\theoremstyle{remark}
\newtheorem{remark}[theorem]{Remark}

\usepackage[capitalise]{cleveref}

\numberwithin{equation}{section}

\newcommand{\im}{\mathrm{im}}
\newcommand{\id}{\mathrm{id}}

\newcommand{\h}{\mathrm{h}}
\newcommand{\U}{\mathrm{U}}

\newcommand{\isogr}{\cong_{\rm gr}}

\newcommand{\supp}{\mathrm{supp}}

\newcommand{\soc}{\mathrm{soc}}

\newcommand{\rad}{\mathrm{rad}}
\newcommand{\radgr}{\mathrm{rad}^{\mathrm{gr}}}

\newcommand{\subgr}{\leq_{\mathrm{gr}}}

\newcommand{\proj}{\mathrm{proj}}

\DeclareMathOperator{\Fun}{Fun}
\newcommand{\Homgr}{\mathrm{Hom}_{\textrm{gr-}R}}

\newcommand{\Hom}{\mathrm{Hom}}
\newcommand{\HOM}{\mathrm{HOM}}
\newcommand{\End}{\mathrm{End}}
\newcommand{\END}{\mathrm{END}}
\newcommand{\Mod}{\mathrm{Mod}}
\newcommand{\fgmod}{\mathrm{mod}}

\newcommand{\M}{\mathrm{M}}
\newcommand{\UT}{\mathrm{UT}}

\begin{document}

\title[Hopkins--Levitzki Type Theorems for Groupoid Graded Rings]{Hopkins--Levitzki Type Theorems for Groupoid Graded Rings}


\author[Z. Cristiano]{Zaqueu Cristiano}
\address{Department of Mathematics - IME, University of São Paulo, Rua do Matão 1010, São Paulo, SP, 05508-090, Brazil}
\email{zaqueucristiano@ime.usp.br}

\author[W. Marques de Souza]{Wellington Marques de Souza}
\address{Institute of Mathematics (INMA), Federal University of Mato Grosso do Sul, Cidade Universitária, Caixa Postal 549, Campo Grande, MS 79070-900, Brazil}
\email{marques.souza@ufms.br}

\author[J. Sánchez]{Javier Sánchez}
\address{Department of Mathematics - IME, University of São Paulo, Rua do Matão 1010, São Paulo, SP, 05508-090, Brazil}
\email{jsanchez@ime.usp.br}

\thanks{This study was financed in part by the Coordenação de Aperfeiçoamento de Pessoal de Nível Superior - Brasil (CAPES) - Finance Code 001. 
This study was financed, in part, by the São Paulo Research Foundation (FAPESP), Brasil. Process Numbers \texttt{\#}2021/14132-2 and \texttt{\#}2020/16594-0.}


\date{August 28, 2026}


\begin{abstract}
    We continue the study of the basic theory of object-unital groupoid graded rings.
    In this work, we are especially interested in nilpotency conditions on the graded Jacobson radical.
    We introduce the concept of left/right objectwise nilpotency of graded ideals, and prove that this condition is appropriate for obtaining graded generalizations of the Hopkins--Levitzki theorem.
    Although this condition is not symmetric, we show that its two-sided version is  suitable for defining gr-semiprimary rings.
    It is known that one-sided $\Gamma_0$-artinian rings need not be $\Gamma_0$-noetherian, but using our tools we prove that one-sided gr-hereditary $\Gamma_0$-artinian rings, two-sided $\Gamma_0$-artinian rings, and $d$-finitely generated one-sided $\Gamma_0$-artinian rings are $\Gamma_0$-noetherian.
    However, the first class need not be  gr-semiprimary, whereas the other two always are.
    We illustrate our results with several (counter)examples, especially involving graded upper triangular matrices.
\end{abstract}

\maketitle

\setcounter{tocdepth}{2} 
\tableofcontents 


\section{Introduction}

The Hopkins--Levitzki theorem plays a central role in ring theory. Beyond the standard formulation that a one-sided artinian ring is noetherian on the same side, it can also be expressed as the condition that, over a semiprimary ring (i.e., a semilocal ring whose Jacobson radical is nilpotent), a module is artinian if and only if it is noetherian (see, e.g., \cite[Theorem~4.15]{Lam1}). Analogs of these results exist for group-graded rings \cite[Section~2.9]{NastasescuVanOystaeyen}. In this work, we address groupoid-graded versions of these classical theorems, adding to an active line of research focused on variants and generalizations of the Hopkins–Levitzki theorem (see, among many others, \cite{Zimmermann, Gera_Patel_Gupta, OstadhadiDehkordi_Shum, PhuongThao_VanSanh}, and  \cite{Albu} for a survey).

Throughout this section, let $\Gamma$ denote a groupoid (i.e., a small category in which every morphism is invertible) whose object set $\Gamma_0$ may be infinite. The $\Gamma$-graded rings $R$ studied here are non-unital in general, but are endowed with a (very good) set of local units $\{1_e\}_{e\in\Gamma_0}$. Every graded right $R$-module $M$ decomposes as $M=\bigoplus_{e\in\Gamma_0}M(e)$, and the ring $R$ itself admits the canonical decompositions $R=\bigoplus_{e\in\Gamma_0}1_eR=\bigoplus_{e\in\Gamma_0}R1_e$ into right and left graded ideals. Consequently, traditional graded artinian and noetherian conditions are too restrictive for this setting. For this reason, we adopt a natural extension to groupoid-graded rings of the notion of categorically artinian (resp., categorically noetherian) modules introduced in \cite[Definition~1.1]{Abrams_ArandaPino_Perera_SilesMolina} and \cite[Section~4.2]{Abrams_Ara_SilesMolina}. Specifically, we say that $M$ is \emph{$\Gamma_0$-artinian} (resp., \emph{$\Gamma_0$-noetherian}) if $M(e)$ is gr-artinian (resp., gr-noetherian) for every $e \in \Gamma_0$. 
Extending the Hopkins--Levitzki theorem to this setting is not straightforward, as there exist one-sided $\Gamma_0$-artinian rings that fail to be $\Gamma_0$-noetherian on the same side \cite[Corollary~3.59]{chainconditions_Jacobsonradical}. Furthermore, the classical notion of nilpotency proves overly restrictive for rings that decompose as infinite direct sums of one-sided ideals. 
The nilpotency type condition that best suits our purposes is what we call \emph{objectwise nilpotency} (see Section~\ref{sec:nilpotency}). Unlike classical nilpotency, this concept is asymmetric; that is, there exist ideals that are right objectwise nilpotent but fail to be left objectwise nilpotent, even for the graded Jacobson radical of one-sided $\Gamma_0$-artinian rings (\Cref{prop: contra-exemplos J fort nilp}). Our Weak Hopkins--Levitzki Theorem asserts that a right $\Gamma_0$-artinian ring $R$ is right $\Gamma_0$-noetherian if and only if its graded Jacobson radical $\radgr(R)$ is right objectwise nilpotent (\Cref{teo: art => (noet <=> fort nilp)}). Defining $R$ as gr-semiprimary when $R/\radgr(R)$ is gr-semisimple and $\radgr(R)$ is objectwise nilpotent on both sides, our Strong Hopkins--Levitzki Theorem proves that $\Gamma_0$-artinianness and $\Gamma_0$-noetherianness are equivalent for every graded $R$-module (\Cref{teo: hopkins-levitski}). 
Furthermore, we provide an example of a graded ring $R$ that satisfies the Weak Hopkins--Levitzki hypotheses but possesses a right $\Gamma_0$-artinian module that is not $\Gamma_0$-noetherian; remarkably, however, every right $\Gamma_0$-noetherian module over such a ring is always $\Gamma_0$-artinian (\Cref{teo: left str. G_0-nilp. ==> right noeth.}). 

Along the way, we also highlight additional conditions under which a one-sided $\Gamma_0$-artinian ring becomes either $\Gamma_0$-noetherian on the same side or gr-semiprimary, namely: $\Gamma_0$-artinianity on the other side (\Cref{teo: G_0-art => G_0-noeth}), strong $\Gamma_0$-artinianity or gr-artinianity (\Cref{coro: fort art => fort noet}), gr-hereditarity  (\Cref{coro: R G0-art gr-hered => G0-noet}), certain graded matrix rings (\Cref{coro:R_gr-art=>M_I(R)_G0-semiprimary}), and $d$-finite generation (\Cref{coro: R=R1_eR G_0-art => rad(R) nilpot}).

Throughout this work, to clarify the underlying concepts and demonstrate the necessity of our hypotheses, we provide several examples and counterexamples, primarily involving rings of graded upper triangular matrices. We conclude by establishing a link between our theoretical framework and the representation theory of algebras.

\medskip

The paper is organized as follows.
In the next section, we present the preliminary definitions and results about groupoids, groupoid graded rings, and groupoid graded modules. 
Moreover, we recall the necessary machinery regarding graded chain conditions and the graded Jacobson radical from \cite{chainconditions_Jacobsonradical}, together with some improvements to the results therein.

In \Cref{sec:nilpotency}, we present the nilpotency conditions on graded ideals that will be used throughout the paper. 
In addition to the standard conditions of nilpotency and gr-nil, we introduce the $\Gamma_0$-nilpotency and left/right objectwise nilpotency, which are specific to the groupoid graded context. 

The graded Jacobson radical of one-sided $\Gamma_0$-artinian rings is studied in \Cref{sec:artinian}. Here, we show that for these rings the graded Jacobson radical is gr-nil and $\Gamma_0$-nilpotent, but need not be nilpotent.
However, we prove that one-sided strongly $\Gamma_0$-artinian rings and one-sided gr-artinian rings have a nilpotent graded Jacobson radical.

The main results of the paper are presented in \Cref{sec:main}.
We prove a weak graded version of the Hopkins--Levitzki theorem, characterizing when a $\Gamma$-graded ring is both right $\Gamma_0$-artinian and right $\Gamma_0$-noetherian, and showing that objectwise nilpotency is a reasonable notion of nilpotency.
Before proving the main theorem, we present an asymmetric graded version of the Hopkins--Levitzki theorem, which shows that one-sided objectwise nilpotency guarantees at least one of the implications between gr-artinianity and gr-noetherianity for graded modules over a gr-semilocal ring.
We then present the strong graded version of the Hopkins--Levitzki theorem, showing that, over a gr-semiprimary ring, the notions of gr-artinian and gr-noetherian modules coincide.
Furthermore, we prove that two-sided $\Gamma_0$-artinian rings are gr-semiprimary.

In order to illustrate the concepts and present some important counterexamples, we study in \Cref{sec:UT_I(A)} the graded upper triangular matrix rings $\UT_I(A)$, where $A$ is a unital ring and $I$ is a partially ordered set.
We characterize the various nilpotency conditions considered here for the graded Jacobson radical of these rings in terms of properties of $A$ and $I$.
Moreover, we determine when these graded rings are gr-semiprimary.
Finally, we exhibit examples of graded rings that are gr-semilocal (indeed one-sided $\Gamma_0$-artinian) whose graded Jacobson radical is one-sided objectwise nilpotent, but for which the notions of gr-artinian and gr-noetherian are not equivalent for graded modules.

In \Cref{sec:gr-hered_gr-art}, we study an important class of graded rings for which right $\Gamma_0$-artinian rings are right $\Gamma_0$-noetherian: the right gr-hereditary rings.
To this end, we establish graded versions of classical results such as the description of the graded Jacobson radical of graded matrix rings and the existence of gr-maximal graded submodules in nonzero gr-projective modules.
We conclude the section characterizing when a graded upper triangular matrix ring is right gr-hereditary and right $\Gamma_0$-artinian, and showing that such graded rings need not be gr-semiprimary.

Finally, we introduce the notion of $d$-finitely generated graded rings in \Cref{sec:d-fg}.
We show that these rings have the interesting property that concepts such as $\Gamma_0$-artinianity, $\Gamma_0$-noetherianity, gr-semisimplicity, gr-semilocality and gr-semiprimarity can be studied by considering small parts of the ring.
Moreover, we prove that the notions of objectwise nilpotency, $\Gamma_0$-nilpotency and nilpotency are equivalent for graded ideals in these rings.
We then use our terminology and results to give a characterization of representation-finite algebras over a field.

\section{Preliminaries}
\label{sec:preliminares}

In this section, we present some basic definitions and results that will be used throughout the paper.

\begin{definition}
    A \emph{groupoid} is the set of morphisms of a small category in which every morphism is an isomorphism.
    For a groupoid $\Gamma$:
    \begin{enumerate}
    \renewcommand{\theenumi}{\roman{enumi}}
        \item we denote by  $\Gamma_0$ the set of all identity (also called \emph{idempotent}) morphisms.
        \item we define, for each $\gamma\in\Gamma$, the elements
        \[d(\gamma):=\gamma^{-1}\gamma\quad\text{and}\quad r(\gamma):=\gamma\gamma^{-1},\]
        which are the identity morphisms corresponding to the domain and the codomain of $\gamma$, respectively.
        \item we say that $\delta\gamma$ \emph{is (not) defined} in $\Gamma$ if   $d(\delta)= r(\gamma)$ (resp.\ $d(\delta)\neq r(\gamma)$).
        \item for $\Sigma,\Sigma'\subseteq \Gamma$, we define
        \[\Sigma\Sigma':=\{\sigma\sigma': \sigma\in \Sigma, \sigma'\in \Sigma' \text{ and } d(\sigma)=r(\sigma')\}\quad\text{and}\quad \Sigma^{-1}:=\{\sigma^{-1} : \sigma\in \Sigma\},\]
        and we denote $\alpha\Sigma':=\{\alpha\}\Sigma'$ and $\Sigma\beta:=\Sigma\{\beta\}$ for all $\alpha,\beta\in\Gamma$.
    \end{enumerate}
\end{definition}

\begin{example}
\label{ex:groupoids}
    Let $G$ be a group and $I$ be a nonempty set. 
    Then $I\times G\times I$ can be regarded as the groupoid consisting of the morphisms of the category $\mathcal{C}$, whose objects are the elements of $I$, morphisms are given by $\hom_{\mathcal{C}}(i,j):=\{j\}\times G\times\{i\}$ for all $i,j\in I$, and composition satisfies $(k,h,j)(j,g,i)=(k,hg,i)$ for all $i,j,k\in I$ and $g,h\in G$.
    This example is important because  any groupoid is a disjoint union of subgroupoids isomorphic (in a non canonical way) to groupoids of this form \cite[p.~125]{Brown}.
    As important particular cases, we have:
    \begin{enumerate}
        \item If $I=\{*\}$ is a singleton, then $I\times G\times I\equiv G$, and therefore every group is a groupoid.
        \item Taking $G$ as the trivial group, then $I\times I$ is a groupoid with $(k,j)(j,i)=(k,i)$ for all $i,j,k\in I$.\qed
    \end{enumerate}
\end{example}

We recommend \cite{Higgins} or \cite{IR} for interested readers in the study of groupoids. 

\bigskip

\emph{Throughout the remainder of the paper, let $\Gamma$ be a groupoid.}

\medskip

\begin{definition}
\label{def:grading_in_additivegroups}
    An additive group $X$ is \emph{$\Gamma$-graded} if there exists a family $\{X_\gamma:\gamma\in\Gamma\}$ of subgroups of $X$ such that $X=\bigoplus_{\gamma\in\Gamma}X_\gamma$.   
    In this case, we fix the following standard notations:
    \begin{enumerate}
    \renewcommand{\theenumi}{\roman{enumi}}
        \item for each $\gamma\in\Gamma$, $X_\gamma$ is called the \emph{homogeneous component of degree $\gamma$ of $X$}
        \item for all $\gamma\in\Gamma$ and $0\neq x\in X_\gamma$, we say that $x$ is \emph{homogeneous of degree $\gamma$} and we write $\deg(x)=\gamma$.
        \item for each $x \in X$, the expression $x = \sum_{\gamma \in \Gamma} x_\gamma$ indicates that $x_\gamma \in X_\gamma$ for all $\gamma \in \Gamma$, with $x_\gamma = 0$ for all but finitely many $\gamma$.
        \item we denote by $\h(X):=\bigcup_{\gamma\in\Gamma}X_\gamma$ the set of \emph{homogeneous elements} of $X$.
        \item we denote by $\supp(X):=\{\gamma\in\Gamma:X_\gamma\neq0\}$ the \emph{support} of $X$. 
        \item if $\sigma,\tau\in\Gamma$ are such that $d(\sigma)\neq r(\tau)$, then we will write $X_{\sigma\tau}=\{0\}$.
    \end{enumerate}
\end{definition}

\subsection{Groupoid graded rings}\label{sec:grupoid graded rings}
\emph{Throughout this work, rings are assumed to be associative but not necessarily unital.}

\begin{definition}
\label{def:groupoid_graded_ring}
    Let $R$ be a ring.
    We say that $R$ is a \emph{$\Gamma$-graded ring} if there is a family $\{R_\gamma\}_{\gamma\in\Gamma}$ of additive subgroups of $R$ such that 
    \begin{enumerate}
        \renewcommand{\theenumi}{\roman{enumi}}
        \item $R=\bigoplus_{\gamma\in\Gamma}R_\gamma$;
        \item $R_\gamma R_\delta\subseteq  R_{\gamma\delta}$ for all $\gamma,\delta\in\Gamma$ (adopting the convention in \Cref{def:grading_in_additivegroups}(vi)). 
    \end{enumerate}
    Following \cite{CLP,CLP2}, we say that $R$ is \emph{object unital}
    if, moreover,
    \begin{enumerate}
        \renewcommand{\theenumi}{\roman{enumi}}
        \setcounter{enumi}{2}
        \item for each $e\in\Gamma_0$, the ring $R_e$ is unital  with identity element $1_e$ such that, for all $\gamma\in \Gamma$ and $a\in R_{\gamma}$, we have $1_{r(\gamma)}a=a1_{d(\gamma)}=a$.
    \end{enumerate}
    Given an object unital $\Gamma$-graded ring $R$, we define the set 
    $$\Gamma_0'(R):=\{e\in\Gamma_0 \colon 1_e\neq 0\}.$$ 
\end{definition}

It follows from \cite[Proposition 2.1.1]{Lund} that the object unital $\Gamma$-graded ring $R=\bigoplus_{\gamma\in\Gamma}R_\gamma$ is unital if and only if $\Gamma_0'(R):=\{e\in\Gamma_0 \colon 1_e\neq 0\}$ is finite. 

\medskip 

\emph{In this work, all groupoid graded rings are supposed to be object unital.}

\medskip

By \Cref{ex:groupoids}, every unital group graded ring is a groupoid graded ring. 
Another key example of a groupoid graded ring is the following, which appeared in \cite{Gabriel} but, to the best of our knowledge, was studied as a groupoid graded ring for the first time in \cite{groupoid_graded_semisimple}. 
        
    \begin{example} \label{ex: anel de categoria}
        Let $\mathcal{C}$ be a small preadditive category, i.e., the class of objects is a set and every class of morphisms is an additive group such that the composition of morphisms is $\mathbb{Z}$-bilinear.
        Denote the set of objects of $\mathcal{C}$ by $\mathcal{C}_0$ and the set of morphisms from $X\in\mathcal{C}_0$ to $Y\in\mathcal{C}_0$ by $\mathcal{C}(X,Y)$. 
        We have a ring
        \[R_\mathcal{C}:=\bigoplus_{X,Y\in \mathcal{C}_0}{\mathcal{C}}(X,Y)\] 
        where, given morphisms $g$ and $g'$, the product $g'$ is defined by $g\circ g'$ if the composition is possible and $0$ otherwise. 
        This ring has a natural $\mathcal{C}_0\times \mathcal{C}_0$-grading via 
        \[(R_\mathcal{C})_{(Y,X)}:={\mathcal{C}}(X,Y)\]
        for all $X,Y\in \mathcal{C}_0$. 
        The identity morphisms of $\mathcal{C}$ play the role of the local units in \Cref{def:groupoid_graded_ring}.\qed
\end{example}

For a generalization of \Cref{ex: anel de categoria} with group graded categories, see \cite[Example~2.3(2)]{groupoid_graded_semisimple}. 
For other examples of groupoid graded rings, see \cite[Example~2.2]{groupoid_graded_semisimple}.
Further examples involving matrices will appear in \Cref{sec:UT_I(A),sec:gr-hered_gr-art}.

\begin{definition}
    Let $R$ be a $\Gamma$-graded ring. 
    A \emph{graded right (left) ideal} $U$ of $R$ is a right (left) ideal of $R$ such that $U=\bigoplus_{\gamma\in\Gamma}U_\gamma$, where $U_\gamma=U\cap R_\gamma$ for each $\gamma\in\Gamma$. 
    We say that $U$ is a \emph{graded ideal} of $R$ if $U$ is a graded right ideal and a graded left ideal of $R$. 
    In this event, $R/U$ is a $\Gamma$-graded ring via $(R/U)_\gamma:=\frac{R_\gamma+U}{U}$ for each $\gamma\in\Gamma$. 
\end{definition}

\begin{definition}
    Given $\Gamma$-graded rings $R$ and $S$, a \emph{gr-homomorphism of rings} is a homomorphism of rings $\varphi:R\to S$ which preserves homogeneous components and satisfies $\varphi(1^R_e)=1^S_e$ for all $e\in\Gamma_0$, where $1^R_e$ (resp. $1^S_e$) denotes the unity of the ring $R_e$ (resp. $S_e$).     A \emph{gr-isomorphism} of $\Gamma$-graded rings is a bijective gr-homomorphism of $\Gamma$-graded rings.     When there exists a gr-isomorphism of $\Gamma$-graded rings $\varphi: R\to S$, we say that $R$ is \emph{gr-isomorphic} to $S$ as rings and we write $R\isogr S$.
\end{definition}

 \emph{For the remainder of this section, let $R=\bigoplus_{\gamma\in \Gamma}R_\gamma$ be a $\Gamma$-graded ring.}

\subsection{Groupoid graded modules}

\begin{definition}
\label{def:graded_module}
    A right $R$-module $M$ is said to be a \emph{$\Gamma$-graded right $R$-module} if there exists a family $\{M_\gamma:\gamma\in\Gamma\}$ of additive subgroups of $M$ such that
    \begin{enumerate}
        \renewcommand{\theenumi}{\roman{enumi}}
        \item $M=\bigoplus_{\gamma\in\Gamma}M_\gamma$
        \item $M_\sigma R_\tau \subseteq  M_{\sigma\tau}$ for all $\sigma,\tau\in\Gamma$ (adopting the convention in \Cref{def:grading_in_additivegroups}(vi))
        \item $m_\sigma1_{d(\sigma)}=m_\sigma$ for all $\sigma\in\Gamma$ and $m_\sigma\in M_\sigma$.
    \end{enumerate}
\end{definition}

By \cite[Proposition 5]{CLP}, assuming (i) and (ii) in \Cref{def:graded_module}, the condition (iii) is equivalent to $MR=M$.

\begin{definition}
    Given a $\Gamma$-graded right $R$-module $M$ and a submodule $N$ of $M$, we say that $N$ is a \emph{graded submodule}, and we denote $N\subgr M$, if $N=\bigoplus_{\gamma\in\Gamma}N_\gamma$ where $N_\gamma:=N\cap M_\gamma$ for each $\gamma\in\Gamma$. 
    In this case, the quotient module $M/N$ is $\Gamma$-graded via $(M/N)_\gamma:=\frac{M_\gamma+N}{N}$ for each $\gamma\in\Gamma$.
\end{definition}

The first example of a groupoid graded module is $R_R$, and its graded submodules are precisely the graded right ideals.

Clearly, any intersection and any sum of graded submodules of a $\Gamma$-graded $R$-module $M$ is a graded submodule of $M$. 
Furthermore, direct sums, graded direct products and graded direct summands of $\Gamma$-graded $R$-modules also are $\Gamma$-graded $R$-modules \cite[Definition~2.12]{chainconditions_Jacobsonradical}.

One of the main constructions of groupoid graded modules is the shift.

\begin{definition}
    Let  $M$ be a $\Gamma$-graded right  $R$-module and $\sigma\in\Gamma$. 
    The \emph{shift  of $M$ by $\sigma$} is the $\Gamma$-graded right $R$-module 
    $$M(\sigma):=\bigoplus_{\gamma\in\Gamma}M(\sigma)_\gamma \quad\text{where } M(\sigma)_\gamma:=\begin{cases}
        M_{\sigma\gamma}, &\text{if } d(\sigma)=r(\gamma)\\
        \{0\}, &\text{if } d(\sigma)\neq r(\gamma).
    \end{cases}$$ 
    By \cite[Lemma 2.6]{groupoid_graded_semisimple}, $M(\sigma)$ equals to $M(r(\sigma))$ as $R$-modules.
    
    Note that $M(e)\subgr M$ for all $e\in\Gamma_0$ and $M=\bigoplus_{e\in\Gamma_0}M(e)$. 
    We denote $$\Gamma'_0(M):=\{e\in\Gamma_0:M(e)\neq0\}.$$
    By \cite[Lemma 2.7]{groupoid_graded_semisimple}, $\Gamma'_0(R_R)=\Gamma'_0(R)$. 
\end{definition}

\begin{definition}
    Let  $M$ and $N$ be $\Gamma$-graded right  $R$-modules.
    \begin{enumerate}
        \item A homomorphism of modules $g:M\to N$  is said to be a \emph{gr-homomorphism} if $g(M_\sigma)\subseteq N_\sigma$ for all $\sigma\in\Gamma$. 
        We denote by $\Homgr(M,N)$ the abelian group consisting of the gr-homomorphisms of modules $g:M\to N$.  
        A bijective gr-homomorphism of modules $g:M\to N$ is called a \emph{gr-isomorphism of modules}.
        When such a gr-isomorphism exists, we say that $M$ is \emph{gr-isomorphic} to $N$ as modules and denote this by $M\cong_{gr}N$.
        
        \item Given $\gamma\in\Gamma$, we say that a homomorphism of right $R$-modules $g:M\to N$ is a \emph{homomorphism of degree $\gamma$} if $g(M_\sigma)\subseteq N(\gamma)_{\sigma}$  for all $\sigma\in\Gamma$.
        In this case, $M(f)\subseteq\ker g$ for all $f\in \Gamma_0\setminus\{d(\gamma)\}$ and $\im\, g\subseteq N(r(\gamma))$ by \cite[Lemma 2.9]{groupoid_graded_semisimple}.
        We denote by $\HOM_R(M,N)_\gamma$ the additive group of the homomorphisms $g:M\to N$ of degree $\gamma$, and we can define the $\Gamma$-graded additive group $\HOM_R(M,N):=\bigoplus_{\gamma\in\Gamma}\HOM_R(M,N)_\gamma$. 
    \end{enumerate}
\end{definition}

\begin{definition}
    Let  $M$ be a $\Gamma$-graded right $R$-module and $e\in\Gamma_0$. 
    By \cite[Proposition 2.10]{groupoid_graded_semisimple}, the canonical projection $\mathds{1}_e:M\to M(e)$ with respect  the decomposition $M=\bigoplus_{e\in\Gamma_0}M(e)$ is the unity of degree $e$ in the $\Gamma$-graded ring $\END_R(M):=\HOM_R(M,M)$.
    Furthermore, $\Gamma'_0(\END_R(M))=\Gamma'_0(M)$ by \cite[Lemma 2.11]{groupoid_graded_semisimple}, and $R\isogr\END(R_R)$ as $\Gamma$-graded rings by \cite[Lemma 3.4]{groupoid_graded_semisimple}.
\end{definition}

\subsection{Graded chain conditions}

We recall some definitions and results from \cite{chainconditions_Jacobsonradical} and extend some of them.

\begin{definition}
\label{def:chain_conditions}
	Let $M$ be a $\Gamma$-graded right $R$-module.
	\begin{enumerate}
		\item We say that $M$ is \emph{gr-artinian} (resp.\ \emph{gr-noetherian}) if $M$ does not contain a strictly decreasing (resp.\ increasing) infinite chain of graded submodules.
        
        \item We say that $R$ is a \emph{right gr-artinian} (resp.\ \emph{right gr-noetherian}) ring if the $R$-module $R_R$ is gr-artinian (resp.\ gr-noetherian).

        \item We say that $M$ is \emph{$\Gamma_0$-artinian} (resp.\ \emph{$\Gamma_0$-noetherian}) if $M(e)$ is a gr-artinian (resp.\ gr-noetherian) $R$-module for each $e\in\Gamma_0$. 

        \item We say that $R$ is a \emph{right $\Gamma_0$-artinian} (resp.\ \emph{right $\Gamma_0$-noetherian}) ring if $R_R$ is a $\Gamma_0$-artinian (resp.\ $\Gamma_0$-noetherian) $R$-module.
        
        \item We say that $M$ is \emph{strongly $\Gamma_0$-artinian} (resp.\ \emph{strongly $\Gamma_0$-noetherian}) if there does not exist a strictly decreasing (resp.\ increasing) infinite chain $\{M_i\}_{i \geq 1}$ of graded submodules of $M$ satisfying the following implication for all $e \in \Gamma_0$ and $i \geq 1$:
        $$M_i(e) = M_{i+1}(e) \implies M_i(e) = M_{i+l}(e) \text{ for all } l \ge 1.$$
        Such a chain is called \emph{tight}. See \cite[Example~3.40]{chainconditions_Jacobsonradical} for examples of these chains.

        \item We say that $R$ is \emph{right strongly $\Gamma_0$-artinian} (resp.\ \emph{right strongly $\Gamma_0$-noetherian}) if $R_R$ is a strongly $\Gamma_0$-artinian (resp.\ strongly $\Gamma_0$-noetherian) $R$-module.	
    \end{enumerate}
\end{definition}

All the chain conditions on groupoid graded modules defined above are related by the following useful result that will be used throughout the text.

\begin{lemma}{\cite[Lemma~3.42]{chainconditions_Jacobsonradical}}
\label{lem:implications_chain_conditions}
Let $M$ be $\Gamma$-graded right $R$-module.
\begin{enumerate}
    \item $M$ gr-artinian $\Rightarrow$ $M$ strongly $\Gamma_0$-artinian $\Rightarrow$ $M$ $\Gamma_0$-artinian.
    \item $M$ gr-noetherian $\Rightarrow$ $M$ strongly $\Gamma_0$-noetherian $\Rightarrow$ $M$ $\Gamma_0$-noetherian.\qed
\end{enumerate}
\end{lemma}

\cite[Example~3.45]{chainconditions_Jacobsonradical} shows that the implications in \Cref{lem:implications_chain_conditions} are not reversible in general.
However, they are reversible when $\Gamma'_0(M)$ is finite \cite[Lemma 3.20]{chainconditions_Jacobsonradical}.

Remember that a nonzero $\Gamma$-graded right $R$-module $M$ is said to be \emph{gr-simple} if its only graded submodules are $\{0\}$ and $M$.

\begin{definition}
\label{def:gr-length}
    Let $M$ be a $\Gamma$-graded right $R$-module.
    \begin{enumerate}
        \item We say that $M$ has \emph{finite gr-length} $n\in\mathbb{N}$ if there exists a finite chain
        \begin{equation*}
        \label{eq: serie de submod}
            M_0=\{0\}\subgr M_1\subgr \cdots \subgr M_n=M,
        \end{equation*}
        where $M_j/M_{j-1}$ is gr-simple for each $1\leq j \leq n$. See \cite[Theorem~3.15]{chainconditions_Jacobsonradical}. 
        
        \item We say that $M$ has \emph{$\Gamma_0$-finite gr-length} if $M(e)$ has finite gr-length for all $e\in\Gamma_0$.
        
        \item We say that $M$ has \emph{strongly $\Gamma_0$-finite gr-length} if there exists $n\in\mathbb{N}$ such that $M(e)$ has finite gr-length less than $n$ for all $e\in\Gamma_0$.
    \end{enumerate}
\end{definition}

We recall the following important result connecting \Cref{def:chain_conditions,def:gr-length}.

\begin{proposition}{\cite[Propositions 3.14, 3.34, and 3.54]{chainconditions_Jacobsonradical}}
\label{prop: comp finito = art + noet}
    Let $M$ be a $\Gamma$-graded right $R$-module.
    The following assertions hold:
    \begin{enumerate}
        \item $M$ has finite gr-length $\Leftrightarrow$ $M$ is both gr-artinian and gr-noetherian.
        \item $M$ has $\Gamma_0$-finite gr-length $\Leftrightarrow$ $M$ is both $\Gamma_0$-artinian and $\Gamma_0$-noetherian. 
        \item $M$ has strongly $\Gamma_0$-finite gr-length $\Leftrightarrow$ $M$ is both strongly $\Gamma_0$-artinian and strongly $\Gamma_0$-noetherian.\qed
    \end{enumerate}
\end{proposition}

The strong $\Gamma_0$-conditions are characterized in \cite[Proposition~3.43]{chainconditions_Jacobsonradical}. This characterization establishes useful connections between the strong $\Gamma_0$-conditions and the $\Gamma_0$-conditions, as shown in the following result, which will allow us to derive results for strong $\Gamma_0$-conditions from corresponding results for $\Gamma_0$-conditions.

\begin{proposition}
\label{prop: comp G-finito => fort art = fort noet}
    Let $M$ be a $\Gamma$-graded right $R$-module.
    \begin{enumerate}
        \item Suppose that $M$ is $\Gamma_0$-noetherian and strongly $\Gamma_0$-artinian. Then $M$ is strongly $\Gamma_0$-noetherian.
        \item Suppose that  $M$ is $\Gamma_0$-artinian and strongly $\Gamma_0$-noetherian. Then $M$ is strongly $\Gamma_0$-artinian.
        \item Suppose that $M$ has $\Gamma_0$-finite gr-length.
        Then $M$ is strongly $\Gamma_0$-artinian if and only if   $M$ is strongly $\Gamma_0$-noetherian.
    \end{enumerate}
\end{proposition}

\begin{proof}
    (1) By the artinian case of \cite[Proposition~3.43]{chainconditions_Jacobsonradical} and \Cref{prop: comp finito = art + noet}(1), $M$ has strongly $\Gamma_0$-finite gr-length. Then, $M$ is strongly $\Gamma_0$-noetherian by \Cref{prop: comp finito = art + noet}(3).

    (2) Analogous to (1).

    (3) This follows from the previous items and \Cref{prop: comp finito = art + noet}(2).
\end{proof}

\subsection{The graded Jacobson radical}

\begin{definition}
	We say that a graded right ideal $\mathfrak{m}$ of $R$ is \emph{gr-maximal} if $\mathfrak{m}$ is maximal in the set of nonzero graded right ideals of $R$.
    We denote by $\radgr(R)$ the \emph{graded Jacobson radical of $R$}, defined as the intersection of all gr-maximal graded right ideals of $R$. 
\end{definition}

We recall some properties of the graded Jacobson radical proved in \cite{chainconditions_Jacobsonradical}.

\begin{theorem}{\cite[Theorem 5.3]{chainconditions_Jacobsonradical}}
\label{teo: a carac de radgr(R)}
    The following assertions are equivalent for $\gamma \in \Gamma$ and $a \in R_{\gamma}$:
    \begin{enumerate}
        \item $a \in \radgr(R)$.
        \item $1_{r(\gamma)} - ax$ is right invertible in $R_{r(\gamma)}$ for each $x \in R_{\gamma^{-1}}$.
        \item $1_{r(\gamma)} - ax$ is invertible in $R_{r(\gamma)}$ for each $x \in R_{\gamma^{-1}}$.\qed
    \end{enumerate}
\end{theorem}

\begin{proposition}
\label{prop:outras_caracs_rad(R)}
    The following statements hold:
    \begin{enumerate}
        \item $\radgr(R)$ is the largest graded ideal $J$ of $R$ such that $1_e+J_e\subseteq \U(R_e)$ for every $e\in\Gamma_0$.
        \item $\radgr(R)$ is the largest graded ideal $J$ of $R$ such that $J_e=\rad(R_e)$ for every $e\in\Gamma_0$.
        \item $\radgr(R)$ is the largest graded ideal $J$ of $R$ such that $1_eJ1_e=\radgr(1_eR1_e)$ for every $e\in\Gamma_0$.
        \item $\radgr(R)$ is the smallest graded ideal $J$ of $R$ such that $\radgr(R/J) = 0$.
    \end{enumerate}
\end{proposition}

\begin{proof}
    Statement (1) is \cite[Corollary~5.4(1)]{chainconditions_Jacobsonradical}, while the statements (2) and (3) are proved in \cite[Corollary~5.6]{chainconditions_Jacobsonradical}.

    We prove item (4). 
    By \cite[Proposition 5.8(4)]{chainconditions_Jacobsonradical}, we have $\radgr(R/\radgr(R))=0$.
    Let $J$ be a graded ideal of $R$ such that $\radgr(R/J) = 0$.
    Consider the canonical gr-homomorphism, $\varphi:R\to R/J$.
    By \cite[Proposition 5.8(1)]{chainconditions_Jacobsonradical}, we have $\varphi(\radgr(R))\subseteq\radgr(R/J)=0$, and it follows that $\radgr(R)\subseteq J$.
\end{proof}

From \Cref{prop:outras_caracs_rad(R)}(1), $\radgr(R)$ is also the intersection of all gr-maximal graded left ideals of $R$.

\begin{definition}
    A $\Gamma$-graded right $R$-module $M$ is called \emph{gr-semisimple} if it is a sum of gr-simple graded submodules.
    By \cite[Proposition~52]{CLP}, every gr-semisimple graded module is a direct sum of gr-simple graded submodules.
    
    We say that $R$ is a \emph{gr-semisimple ring} if $R_R$ is gr-semisimple. 
    By \cite[Corollary~5.28]{groupoid_graded_semisimple}, $R_R$ is gr-semisimple if and only if $_RR$ is gr-semisimple.
    
    We say that $R$ is a \emph{gr-semilocal ring} if the graded ring $R/\radgr(R)$ is gr-semisimple. 
\end{definition}

One-sided $\Gamma_0$-artinian rings are the main examples of gr-semilocal rings. Further examples can be found in \cite[Proposition~6.16]{chainconditions_Jacobsonradical}.

\begin{proposition}{\cite[Proposition 6.2(1)]{chainconditions_Jacobsonradical}}
\label{prop: R gr-art => R/rad(R) gr-ss}
    If $R$ is a right $\Gamma_0$-artinian ring, then $R$ is a  gr-semilocal ring.\qed
\end{proposition}

We recall the following useful characterizations of (strongly) ($\Gamma_0$-)finite gr-length for graded modules over a gr-semilocal ring.

\begin{proposition}{\cite[Corollary~6.7]{chainconditions_Jacobsonradical}}
\label{prop: comp finito <=> MJn=0}
    Suppose that $R$ is gr-semilocal, $M$ is a $\Gamma$-graded right $R$-module, and $J=\radgr(R)$. The following assertions are equivalent:
    \begin{enumerate}
        \item $M$ has finite gr-length.
        \item $M$ is gr-artinian and $MJ^n=0$ for some $n\in\mathbb{N}$.
        \item $M$ is gr-noetherian and $MJ^n=0$ for some $n\in\mathbb{N}$.\qed
    \end{enumerate}
\end{proposition}

\begin{proposition}{\cite[Corollary~6.8]{chainconditions_Jacobsonradical}}
\label{prop: G0 comp finito <=> MJn=0}
    Suppose that $R$ is gr-semilocal, $M$ is a $\Gamma$-graded right $R$-module, and $J=\radgr(R)$. The following assertions are equivalent:
    \begin{enumerate}
        \item $M$ has  $\Gamma_0$-finite gr-length.
        \item $M$ is $\Gamma_0$-artinian and, for each $e\in\Gamma_0$, $M(e)J^{n_e}=0$ for some $n_e\in\mathbb{N}$.
        \item $M$ is $\Gamma_0$-noetherian and, for each $e\in\Gamma_0$,  $M(e)J^{n_e}=0$ for some \mbox{$n_e\in\mathbb{N}$}.\qed
    \end{enumerate}
\end{proposition}

\begin{proposition}{\cite[Corollary~6.9]{chainconditions_Jacobsonradical}}
\label{prop: fort G0 comp finito <=> MJn=0}
    Suppose that $R$ is gr-semilocal, $M$ is a $\Gamma$-graded right $R$-module, and $J=\radgr(R)$. The following assertions are equivalent:    
    \begin{enumerate}
        \item $M$ has strongly $\Gamma_0$-finite gr-length.
        \item $M$ is strongly $\Gamma_0$-artinian and $MJ^n=0$ for some $n\in\mathbb{N}$.
        \item $M$ is strongly $\Gamma_0$-noetherian and $MJ^n=0$ for some $n\in\mathbb{N}$.\qed
    \end{enumerate}
\end{proposition}



\section{Nilpotency-type conditions}
\label{sec:nilpotency}

\begin{definition}\label{def: nilpotency conditions}
Let $R$ be a $\Gamma$-graded ring. We say that $U\subseteq R$ is a \emph{graded subrng} of $R$ if it satisfies the following three conditions:
\begin{enumerate}
    \renewcommand{\theenumi}{\roman{enumi}}
    \item $U+U\subseteq U$
    \item $U\cdot U\subseteq U$
    \item If $x_1+\dotsb+x_n\in U$ with $x_i\in R_{\gamma_i}$ for different $\gamma_1,\dotsc,\gamma_n\in \Gamma$, then $x_1,\dotsc,x_n\in U$.    
\end{enumerate}

Important examples of graded subrngs are right, left and two-sided  graded ideals.
In this context,  
\begin{enumerate}
    \item We say that $U$ is \emph{gr-nil} if every homogeneous element of $U$ is nilpotent.
    \item We say that $U$ is \emph{nilpotent} if there exists a positive integer $n$ such that $U^n = 0$.  
    \item  We say that $U$ is \emph{right $\Gamma_0$-nilpotent} if $1_eU$ is nilpotent for each $e\in\Gamma_0$, i.e. there exists a positive integer $n_e$ such that $(1_eU)^{n_e} = 0$ for each $e\in\Gamma_0$. And we say that  $U$ is \emph{left $\Gamma_0$-nilpotent} if, for each $e \in \Gamma_0$, there exists a positive integer $n_e$ such that  $(U1_e)^{n_e} = 0$. 
    \item We say that $U$ is \emph{right objectwise nilpotent} if, for each $e \in \Gamma_0$, there exists a positive integer $n_e$ such that  $1_eU^{n_e} = 0$. And we say that  $U$ is \emph{left objectwise nilpotent} if, for each $e \in \Gamma_0$, there exists a positive integer $n_e$ such that   $U^{n_e}1_e = 0$.    If $U$ is both left and right objectwise nilpotent, then $U$ is said to be \emph{objectwise nilpotent}.
\end{enumerate}
\end{definition}

Now we provide more practical versions of the  concepts of nil and $\Gamma_0$-nilpotent graded right ideals. Moreover, \Cref{lem: eqiv nil e G_0-nilpotent}(3)  shows that being $\Gamma_0$-nilpotent is a symmetric condition, so we will omit the distinction between ‘left’ and ‘right’.  On the other hand,  there are (two-sided)  graded ideals $U$ that are  left objectwise nilpotent, though not right objectwise nilpotent, as \Cref{prop: contra-exemplos J fort nilp}(4) shows. We sometimes use the next result without further reference to it.

\begin{lemma}\label{lem: eqiv nil e G_0-nilpotent}
Let $U$ be a graded subrng of the $\Gamma$-graded ring $R$. The following statements hold:
\begin{enumerate}
    \item  $U$ is gr-nil if and only if $1_eU1_e$ is gr-nil for each $e\in\Gamma_0$.
    \item $U$ is right $\Gamma_0$-nilpotent if and only if, for each $e \in \Gamma_0$, there exists a positive integer $n_e$ such that $(1_eU1_e)^{n_e} = 0$.    
    \item For each $e\in\Gamma_0$, there exists a positive integer $n$ such that $(1_eU)^n=0$ if and only if there exists a positive integer $m$ such that $(U1_e)^m=0$. 
\end{enumerate}
\end{lemma}

\begin{proof}
(1) Notice that if $a\in U_\gamma$ for some $\gamma\in\Gamma$ with $d(\gamma)\neq r(\gamma)$, then $a^2=0$.

(2) For each $e\in\Gamma_0$ and integer $n\geq2$, we have
\begin{equation}\label{eq: 1eU1e nilp => U G0-nilp1}
(1_eU1_e)^n\subseteq (1_eU)^n = (1_eU)(1_e1_eU)^{n-1} = (1_eU1_e)^{n-1}1_e U,
\end{equation}
from which
\begin{equation}
\label{eq: 1eU1e nilp => U G0-nilp2}
    (1_eU1_e)^{n-1}=0\implies (1_eU)^n=0\implies (1_eU1_e)^n=0.
\end{equation}
Note that the last implication also holds for $n=1$.

(3) Note that there are versions of \eqref{eq: 1eU1e nilp => U G0-nilp1} and \eqref{eq: 1eU1e nilp => U G0-nilp2} for $U1_e$, $e\in\Gamma_0$.
\end{proof}

While the nilpotency-type conditions given in \Cref{def: nilpotency conditions} are related as in \Cref{lem: implications nilpotency conditions} below, they are all distinct, as will be demonstrated in \Cref{prop: nilpotencias radgr UT(A)}.

\begin{lemma}\label{lem: implications nilpotency conditions}
 Let $U$ be a graded subrng of the $\Gamma$-graded ring $R$. Consider the following statements. 
 \begin{enumerate}
    \renewcommand{\theenumi}{\alph{enumi}}
     \item\label{item: nilp} $U$ is nilpotent.
     \item \label{item: strongly G_0-nilp} $U$ is left or right objectwise nilpotent.
     \item \label{item: G_0-nilp} $U$ is $\Gamma_0$-nilpotent.
     \item\label{item: nil} $U$ is gr-nil.
 \end{enumerate}
 Then \eqref{item: nilp}$\Rightarrow$\eqref{item: strongly G_0-nilp}$\Rightarrow$\eqref{item: G_0-nilp}$\Rightarrow$\eqref{item: nil}.
\end{lemma}

\begin{proof}
Clearly, if $U$ is (left or right) $\Gamma_0$-nilpotent, then $U$ is gr-nil. The result now follows because 
 \[(U1_e)^n\subseteq U^n1_e\subseteq U^n \quad\text{ and }\quad (1_eU)^n\subseteq 1_eU^n\subseteq U^n\]
for all $e\in\Gamma_0$ and $n\in\mathbb{N}$.   
\end{proof}

In certain situations, some of the implications in \Cref{lem: implications nilpotency conditions} are reversible.  For example, we will see in \Cref{teo: R G_0-art => rad(R) G_0-nilpotente}(4) that a gr-nil graded right ideal in a  right $\Gamma_0$-artinian ring is $\Gamma_0$-nilpotent. The second part of  \Cref{lem: soma de nilpotentes} offers yet another example of this phenomenon.

\begin{lemma} 
\label{lem: soma de nilpotentes}
Let $R$ be a $\Gamma$-graded ring and let $U_1,U_2, \dots, U_m$ be nilpotent graded right  ideals of $R$.  
Then $U_1 +U_2+ \cdots + U_m$ is nilpotent.
In particular, if $\Gamma'_0(U)$ is finite and $U$ is $\Gamma_0$-nilpotent, then $U$ is nilpotent.
\end{lemma}

\begin{proof}
It is well known that  a finite sum of nilpotent one-sided ideals is nilpotent, see for example \cite[Lemma 4.10]{Lam1}.

The final part is clear from the decomposition $U=\bigoplus_{e\in\Gamma'_0(U)}U(e)$.
\end{proof}

For the next result we will need the following definition from \cite{chainconditions_Jacobsonradical}. A family $\{M_i:i\in I\}$  of $\Gamma$-graded right $R$-modules is \emph{$\Gamma_0$-finite} if the set
 $I_e:=\{i\in I: M_i(e)\neq0\}$ is finite for all $e\in\Gamma_0$. 

\begin{lemma} 
\label{lem: soma de G_0-nilpotentes}
Let $R$ be a $\Gamma$-graded ring and let $\{U_i:i\in I\}$ be a $\Gamma_0$-finite family of $\Gamma_0$-nilpotent graded right ideals. Then $U=\sum_{i\in I} U_i$ is $\Gamma_0$-nilpotent.
\end{lemma}

\begin{proof}
Let $e\in\Gamma_0$. Since $I_e:=\{i\in I: U_i(e)\neq0\}$ is finite, there exist $i_1,\dotsc, i_n\in I$ such that $U(e)=U_{i_1}(e)+\dotsb+U_{i_n}(e)$. Since each $U_{i_k}(e)$ is nilpotent, \Cref{lem: soma de nilpotentes} implies that $U(e)$ is nilpotent.
\end{proof}

\begin{lemma}
    Let $R$ be a $\Gamma$-graded ring and let $U_1,U_2, \dots, U_m$ be right objectwise nilpotent graded  ideals of $R$.  
Then $U_1 +U_2+ \cdots + U_m$ is right objectwise nilpotent.
\end{lemma}

\begin{proof}
The argument is very similar to \cite[Lemma 4.10]{Lam1}.
Let $e\in\Gamma_0.$
Suppose that $n_1,\dots,n_m\in\mathbb{N}$ are such that $U_i^{n_i}(e)=0$ for each $i=1,\dots,m$.
Note that if $a_1,a_2\dots,a_{n_1+n_2}\in U_1$ and $b_1,b_2\dots,b_{n_1+n_2}\in U_2$ with $a_1\in U_1(e)$ and $b_1\in U_2(e)$, then
\begin{equation}
\label{eq: soma de nilpotentes}
   (a_1+b_1)(a_2+b_2)\cdots(a_{n_1+n_2}+b_{n_1+n_2})
\end{equation}
is a sum of terms of the form $c_1c_2\cdots c_{n_1+n_2}$
where $c_i\in\{a_i,b_i\}$ for each $i$.
In each of these terms, the first element belongs to $R(e)$ and there will be at least $n_1$ factors from $U_1$ or else at least $n_2$ factors from $U_2$.
Since $U_1$ and $U_2$ are graded  ideals of $R$ such that $U_1^{n_1}(e)=U_2^{n_2}(e)=0$, it follows that the product \eqref{eq: soma de nilpotentes} equals to zero.
Therefore, $(U_1+U_2)^{n_1+n_2}(e)=0$.
Inductively, we get $(U_1+U_2+\cdots+U_m)^{n_1+n_2+\cdots+n_m}(e)=0$.  
\end{proof}

\begin{lemma} 
\label{lem: soma de gr-nil}
Let $R$ be a $\Gamma$-graded ring and let $\{U_i:i\in I\}$ be a  family of gr-nil graded ideals. Then $U=\sum_{i\in I} U_i$ is gr-nil.
\end{lemma}

\begin{proof}

It suffices to prove that if $U_1$ and $U_2$ are gr-nil graded ideals, then $U_1+U_2$ is gr-nil.
The proof of this is the same as in the ungraded case; see, for example, \cite[Lemma~10.25]{Lam1}.
\end{proof}

We now state some results concerning gr-nil graded ideals and the graded Jacobson radical of graded rings.

\begin{proposition}
\label{prop: radgr contem todos os gr-nil}
Let $R$ be a $\Gamma$-graded ring and $U$ be a graded right ideal of $R$
 If $U_e$ is a nil right ideal of $R_e$ for every $e\in\Gamma_0$. then $U \subseteq \radgr (R)$.  In particular,  $\radgr (R)$ contains every gr-nil graded right ideal of $R$.
\end{proposition}

\begin{proof}
Let $\gamma \in \Gamma$ and nonzero $a \in U_{\gamma}$.  
Then, for each $x \in R_{\gamma^{-1}}$, $ax \in U_{r(\gamma)}$ is nilpotent, say $(ax)^n = 0$ for some positive integer $n$.  
Therefore, $1_{r(\gamma)} - ax$ is invertible in $R_{r(\gamma)}$ with inverse $\sum_{i = 0}^{n-1} (ax)^i$.  
Hence, $a \in \radgr(R)$ by \Cref{teo: a carac de radgr(R)}(3).
\end{proof}

\begin{corollary}
\label{coro: J gr-nil e rad(R/J)=0 => J=rad(R)}
    Let $R$ be a $\Gamma$-graded ring.
    If $J$ is a graded ideal of $R$ such that $J_e$ is nil for every $e\in\Gamma_0$ and $\radgr(R/J)=0$, then $J=\radgr(R)$.
\end{corollary}

\begin{proof}
    If $J_e$ is a nil ideal for every $e\in\Gamma_0$, then $J\subseteq\radgr(R)$ by \Cref{prop: radgr contem todos os gr-nil}. 
    If $\radgr(R/J)=0$, then it follows from \Cref{prop:outras_caracs_rad(R)}(4) that $\radgr(R)\subseteq J$.
\end{proof}

\begin{proposition}
\label{prop: rad(R) = conj dos a tq ax nilpot}
    Let $R$ be a $\Gamma$-graded ring such that $\rad(R_e)$ is a nil ideal of $R_e$ for every $e\in\Gamma_0$.
    Then, for each $\gamma\in\Gamma$, we have
    \begin{align*}
        \radgr(R)_\gamma&=\{a\in R_\gamma: ax \text{ is nilpotent for every } x\in R_{\gamma^{-1}}\}\\
        &=\{a\in R_\gamma: ya \text{ is nilpotent for every } y\in R_{\gamma^{-1}}\}.
    \end{align*}
\end{proposition}

\begin{proof}
    We only prove the first equality; the second follows by symmetry.

    The inclusion $(\supseteq)$ holds by \Cref{prop: radgr contem todos os gr-nil} applied to $U=aR$ for each $a$ that belongs to the set on the right-hand side of the equality.
    For the other inclusion, note that if $a \in \radgr(R)_{\gamma}$ and $x \in R_{\gamma^{-1}}$, then $ax \in \radgr(R)_{r(\gamma)}=\rad(R_{r(\gamma)})$ by \Cref{prop:outras_caracs_rad(R)}(2).
\end{proof}

The following consequence of \Cref{prop: rad(R) = conj dos a tq ax nilpot}  can be regarded as a generalization of \cite[Corollary~II.1.18]{Coelho}.

\begin{corollary}
\label{coro: rad of the category of finite length mod}
    Let $A$ be a unital ring and $\Mod$-$A$ be the category of all right $A$-modules.
    Suppose $\mathcal{C}$ is a small full subcategory of $\Mod$-$A$ such that all objects of $\mathcal{C}$ are modules of finite length. 
    The following assertions are equivalent for $M,N\in\mathcal{C}_0$ and $g\in\Hom_A(M,N)$:
    \begin{enumerate}
        \item $g\in\radgr(R_\mathcal{C})_{(N,M)}$.
        \item $gg'$ is nilpotent for every $g'\in\Hom_A(N,M)$. 
        \item $g'g$ is nilpotent for every $g'\in\Hom_A(N,M)$. 
    \end{enumerate}
\end{corollary}

\begin{proof}
    This follows from \Cref{prop: rad(R) = conj dos a tq ax nilpot}, since $\rad((R_\mathcal{C})_{(M,M)})=\rad(\End_A(M))$ is nilpotent for every $M\in\mathcal{C}_0$ \cite[Exercise 21.24]{LamExercises}.
\end{proof}

\section{Graded Jacobson radical of one-sided \texorpdfstring{$\Gamma_0$}{Gamma0}-artinian rings}
\label{sec:artinian}

In what follows, we turn our attention to the graded Jacobson radical of right $\Gamma_0$-artinian rings.
The following example demonstrates that the graded Jacobson radical need not be nilpotent for this class of rings. 
It also provides an example of a graded ring whose graded Jacobson radical is right and left objectwise nilpotent, but not nilpotent.

\begin{example}\label{ex: radical not nilpotent for G_0-artinian}
    For each integer $n>1$, let $A_n$ be a two-sided artinian ring such that $(\rad(A_n))^n=0$ and  $(\rad(A_n))^{n-1}\neq0$ (for example, $A_n$ being the ring of $n\times n$ upper triangular matrices over a semisimple ring).
    Consider $R:=\bigoplus_{n>1}A_n$ with the obvious $\mathbb{N}\times\mathbb{N}$-grading with support $\{(n,n):n>1\}$.
    Then $R$ is a two-sided $\Gamma_0$-artinian ring such that $\radgr(R)$ is not nilpotent, since $(\radgr(R))^n\supseteq(\rad(A_{n+1}))^n\neq0$. If we set $e_n=(n,n)$, then $(\radgr(R))^n(e_n)=(e_n)(\radgr(R))^n=\rad(A_{n})^n=0$. \qed
\end{example}

Although not nilpotent, the graded Jacobson radical of a  one-sided $\Gamma_0$-artinian ring satisfies the following nilpotency-type conditions.

\begin{theorem} 
\label{teo: R G_0-art => rad(R) G_0-nilpotente}
The following statements hold for a right  $\Gamma_0$-artinian ring $R$:
\begin{enumerate}
    \item $\radgr(R)$ is the greatest $\Gamma_0$-nilpotent graded (right, left) ideal of $R$.
    \item $\radgr(R)$ is the greatest gr-nil graded (right, left) ideal of $R$.
    \item  If $U$ is a graded (right, left)  ideal of $R$ such that $U_e$ is a nil (right, left)  ideal of $R_e$ for every $e\in\Gamma_0$, then $U$ is $\Gamma_0$-nilpotent.
    \item Every gr-nil graded (right, left) ideal of $R$ is $\Gamma_0$-nilpotent.
\end{enumerate}
\end{theorem}

\begin{proof}
We prove (1) and (2) at the same time.

Since any gr-nil (right, left) ideal of $R$ is contained in $\radgr(R)$  by \Cref{prop: radgr contem todos os gr-nil}, and $\Gamma_0$-nilpotence implies the gr-nil property, it is enough to show that $\radgr(R)$ is $\Gamma_0$-nilpotent.

Let $e\in\Gamma_0$.
By \cite[Lemma~3.9(2)]{chainconditions_Jacobsonradical}, $1_eR1_e$ is a right gr-artinian $e\Gamma e$-graded ring. 
It is shown in the proof of \cite[Corollary 2.9.7]{NastasescuVanOystaeyen} that $\radgr(1_eR1_e)$ is nilpotent.
Therefore, by \Cref{prop:outras_caracs_rad(R)}(3), $1_e\radgr(R)1_e$ is nilpotent for each $e\in\Gamma_0$.
Hence, $\radgr(R)$ is $\Gamma_0$-nilpotent.

(3) By \Cref{prop: radgr contem todos os gr-nil}, $U\subseteq\radgr(R)$, and thus $U$ is $\Gamma_0$-nilpotent by (1).

(4) Follows from (3).
\end{proof}

\begin{corollary}
\label{coro: rad(R) como unico nil com quoc gr-ss}
    Let $R$ be a right $\Gamma_0$-artinian ring and $J$ be a graded ideal of $R$. 
    Then $J=\radgr(R)$ if and only if $J$ is gr-nil and $R/J$ is a gr-semisimple ring.
\end{corollary}

\begin{proof}
    $\radgr(R)$ is gr-nil by \Cref{teo: R G_0-art => rad(R) G_0-nilpotente}(2), and $R/\radgr(R)$ is a gr-semisimple ring by \Cref{prop: R gr-art => R/rad(R) gr-ss}.

    If $R/J$ is a gr-semisimple ring, then $\radgr(R/J)=0$ by \cite[Theorem 5.11]{chainconditions_Jacobsonradical}. 
    Therefore, the result follows from \Cref{coro: J gr-nil e rad(R/J)=0 => J=rad(R)}.
\end{proof}

As we will see in \Cref{prop: contra-exemplos J fort nilp}, the conclusion of \Cref{teo: R G_0-art => rad(R) G_0-nilpotente}(1) is optimal, in the sense that there exist one-sided $\Gamma_0$-artinian rings whose graded Jacobson radical is neither left nor right objectwise nilpotent. Moreover, \Cref{prop: contra-exemplos J fort nilp} demonstrates that there exist one-sided $\Gamma_0$-artinian rings whose graded Jacobson radical is one-sided objectwise nilpotent, but not objectwise nilpotent on the opposite side.

Under additional conditions on $\radgr(R)$ or $R$, the conclusions of \Cref{teo: R G_0-art => rad(R) G_0-nilpotente}(1) can be improved.
\begin{proposition}\label{prop:nilpotence of rad(R)}
Let $R$ be a $\Gamma$-graded ring. The following statements hold:
\begin{enumerate}
    \item Suppose that $R$ is right $\Gamma_0$-artinian and that there exists a positive integer $k$ such that $\radgr(R)^k=\radgr(R)^{k+1}$, then $\radgr(R)$ is the largest nilpotent graded left ideal and the largest nilpotent graded right ideal.
    \item If $R$ is right strongly $\Gamma_0$-artinian, then $\radgr(R)$ is  the largest nilpotent graded left ideal and the largest nilpotent graded right ideal.

    \item If $R$ is right gr-artinian, then $\radgr(R)$ is  the largest nilpotent graded left ideal and the largest nilpotent graded right ideal.
\end{enumerate}
\end{proposition}

\begin{proof}
Set $J:=\radgr(R)$.

(1) Since every nilpotent graded right (left) ideal is gr-nil, it follows from \Cref{prop: radgr contem todos os gr-nil} that it suffices to prove that $J$ is nilpotent.

Observe that $J^k = J^{k+l}$ for every $l \geq 0$. Set  $ I:=J^k$.
Suppose that $I(e) \neq 0$ for some $e \in \Gamma_0$.
Since $R$ is a right $\Gamma_0$-artinian ring, we can choose a minimal element $U_0$ in $\{U \subgr R(e) : UI \neq 0\}$.
Take $x \in (U_0)_\gamma$, for some $\gamma \in \Gamma$, such that $xI \neq 0$.
By the idempotency of $I$ and the minimality of $U_0$, we have $xI = U_0$.
Thus, there exists $a \in I_{d(\gamma)}$ such that $x = xa$.
Hence, $x(1_{d(\gamma)} - a) = 0$.
Since $a \in \radgr(R)_{d(\gamma)}$, it follows that $1_{d(\gamma)} - a$ is invertible in $R_{d(\gamma)}$, which implies that $x = 0$, a contradiction.
Therefore, $I = 0$, which shows that $J$ is nilpotent.

(2) The descending chain
$J \supseteq J^2 \supseteq J^3 \supseteq \cdots \supseteq J^n \supseteq \cdots$
is tight.
Since $R_R$ is right strongly $\Gamma_0$-artinian, there exists a positive integer $k$ such that $J^k = J^{k+l}$ for every $l \geq 0$. Hence the result follows from (1).

(3) If $R$ is right gr-artinian, then $R$ is right strongly $\Gamma_0$-artinian by \cite[Lemma 3.42]{chainconditions_Jacobsonradical}. Hence $\radgr(R)$ is nilpotent by (2). We observe that the nilpotence of $\radgr(R)$ could also be shown as follows. By \cite[Lemma~3.20]{chainconditions_Jacobsonradical}, $R$ is a right $\Gamma_0$-artinian ring and $\Gamma'_0(R)$ is a finite set. 
Then $\radgr(R)$ is $\Gamma_0$-nilpotent by \Cref{teo: R G_0-art => rad(R) G_0-nilpotente}, and it follows that $\radgr(R)$ is nilpotent by \Cref{lem: soma de nilpotentes}.
\end{proof}

For the proofs that follow, we will need an auxiliary result involving  the concepts in the following definition. 

\begin{definition}
    Let $R$ be a $\Gamma$-graded ring and $M$ a $\Gamma$-graded right  $R$-module. 
    
    We say that $M$ is \emph{$\Gamma_0$-finitely generated} if $M(e)$ is finitely generated for all $e \in \Gamma_0$.

    We say that $M$ is \emph{gr-cyclic} if $M = mR$ for some $m \in h(M)$, and \emph{$\Gamma_0$-gr-cyclic} if $M(e)$ is gr-cyclic for all $e \in \Gamma_0$.
\end{definition}

\begin{lemma}
\label{lema: M fg e J G0-fg => MJ fg}
    Let $R$ be a $\Gamma$-graded ring, $M$ be a $\Gamma$-graded right $R$-module, and $J$ be a graded ideal of $R$.
    If $M$ is finitely generated (resp. gr-cyclic) and $J$ is $\Gamma_0$-finitely generated (resp. $\Gamma_0$-gr-cyclic) as a right $R$-module, then $MJ$ is finitely generated (resp. gr-cyclic).
    In particular, if $J$ is $\Gamma_0$-finitely generated (resp. $\Gamma_0$-gr-cyclic) as a right $R$-module, then $J^n$  is also for all $n\geq1$.
\end{lemma}

\begin{proof}
    Suppose that $M=m_1R+\cdots+m_nR$ for some $m_i\in M_{\gamma_i}$, $\gamma_i\in\Gamma$, and $J(d(\gamma_i))=x_{i1}R+\cdots+x_{in_i}R$ for each $i=1,\dots,n$.
    Then
    \[MJ\subseteq \sum_{i=1}^n m_iRJ=\sum_{i=1}^n m_iJ=\sum_{i=1}^n m_iJ(d(\gamma_i))\subseteq\sum_{i=1}^n\sum_{k=1}^{n_i} m_ix_{ik}R\subseteq MJ.\]
    If $M=mR$ for some $m\in M_\gamma$ and $J(d(\gamma))=xR$ for some $x\in\h(R)$, then $MJ=mJ=mJ(d(\gamma))=mxR$ is gr-cyclic.

    In particular, if $J$ is $\Gamma_0$-finitely generated (resp. $\Gamma_0$-gr-cyclic) as a right $R$-module and $e\in\Gamma_0$, then $J^2(e)=J(e)\cdot J$ is finitely generated (resp. gr-cyclic), $J^3(e)=J^2(e)\cdot J$ is finitely generated (resp. gr-cyclic), and so on.
\end{proof}

We end this subsection with an important object in the study of $\Gamma_0$-artinian rings and their graded Jacobson radicals.

\begin{definition}
    Let $R$ be a $\Gamma$-graded ring and $J$ a graded subrng of $R$. We define $J^\infty$ as the graded subrng $J^\infty=\bigcap_{n\geq 1}J^n$.
\end{definition}

\begin{proposition}\label{prop: J infinity}
Let $R$ be a $\Gamma$-graded ring. Suppose that $R$ is a right $\Gamma_0$-artinian ring and set $J:=\radgr(R)$. The following assertions hold:
\begin{enumerate}
    \item  For each $e\in\Gamma_0$, there exists a positive integer $n_e$ such that $J^\infty(e)=J^{n_e+l}(e)$ for every $l\geq0$. 
    \item $J^\infty\cdot J=J^\infty$.
    \item $J$ is right objectwise nilpotent if and only if $J^\infty=\{0\}$.
    \item  If $J^k$ is $\Gamma_0$-finitely generated as a right ideal for some  $k\in\mathbb{N}\cup\{\infty\}$, then $J$ is right objectwise nilpotent.
    \item $\radgr\left(R/J^\infty\right)=J/J^\infty$ is right objectwise nilpotent.
\end{enumerate}
\end{proposition}

\begin{proof}
(1) For each $e\in\Gamma_0$, consider the descending chain of right ideals 
\[
J(e)\supseteq J^2(e)\supseteq \dotsb \supseteq J^r(e)\supseteq \dotsb
\]
Since $R$ is right $\Gamma_0$-artinian, there exists a positive integer $n_e$ such that $J^{n_e}(e)=J^{n_e+l}(e)$ for all $l\geq 0$.
Therefore, $J^\infty(e)=J^{n_e}(e)$ 
for each $e\in\Gamma_0$. 

(2) $J^\infty=J^\infty\cdot J$ because $J^{n_e}(e)\cdot J=J^{n_e+1}(e)=J^{n_e}(e)$ 
for each $e\in \Gamma_0$. 

(3) By (1), $J^\infty=\{0\}$ if and only if $J$ is right objectwise nilpotent. 

(4)  First, we show that if $J^k$ is $\Gamma_0$-finitely generated as a right ideal for some positive integer $k$, then $J^\infty$  is also $\Gamma_0$-finitely generated.
For each $e\in\Gamma_0$, let $n_e$ be a positive integer as in (1). Now let $k_e$ be a positive integer such that $kk_e\geq n_e$ for each $e\in\Gamma_0$. 
Since $J^k$ is $\Gamma_0$-finitely generated,  \Cref{lema: M fg e J G0-fg => MJ fg} implies that $J^\infty(e)=J^{kk_e}(e)=(J^{k})^{k_e}(e)$ is finitely generated for each $e\in\Gamma_0$. Therefore, $J^\infty$ is $\Gamma_0$-finitely generated. From (2) and the graded Nakayama's Lemma \cite[Proposition~5.12]{chainconditions_Jacobsonradical}, we get that $J^\infty=\{0\}$.

(5) By \cite[Proposition 5.8(3)]{chainconditions_Jacobsonradical}, $\radgr(R/J^\infty)=J/J^\infty$. Now
    \[\radgr(R/J^\infty)^{n_e}(e)=(J^{n_e}/J^\infty)(e)\isogr J^{n_e}(e)/J^\infty(e)=0.\]
    Therefore, $\radgr(R/J^\infty)$ is right objectwise nilpotent.    
\end{proof}

    Recall the following definitions from \cite[Definition~4.29]{chainconditions_Jacobsonradical}.

    \begin{definition}
        Let $R$ be a $\Gamma$-graded ring and $M$ be a $\Gamma$-graded right $R$-module. 
     We define the \emph{gr-socle} $\soc_{\rm gr}^1(M)$ of $M$ as the sum of all gr-simple graded submodules of $M$.
    If there does not exist a gr-simple graded submodule of $M$, we set $\soc_{\rm gr}^1(M)=0$.
    Set $\soc_{\rm gr}^0(M)=\{0\}$. Inductively, for each $n\in\mathbb{N}$, we define $\soc_{\rm gr}^{n+1}(M)$ as the graded submodule of $M$ containing $\soc_{\rm gr}^n(M)$ such that 
    $$\frac{\soc_{\rm gr}^{n+1}(M)}{\soc_{\rm gr}^n(M)}=\soc_{\rm gr}\left(\frac{M}{\soc_{\rm gr}^n(M)}\right).$$ 

    We have the following chain  for each graded right $R$-module $M$:
\begin{equation*}
\label{eq: socle series}    
    0=\soc_{\rm gr}^0(M)\subseteq \soc_{\rm gr}^1(M)\subseteq\soc_{\rm gr}^2(M)\subseteq\cdots\subseteq\soc_{\rm gr}^n(M)\subseteq\cdots\subseteq M
\end{equation*}

We define $\soc_{\rm gr}^{\infty}(M)=\bigcup_{n\geq 0}\soc_{\rm gr}^n(M).$
\end{definition}

\begin{proposition}\label{prop: soc infty}
Let $R$ be a $\Gamma$-graded ring. Suppose that $R$ is a right $\Gamma_0$-artinian ring and set $J:=\radgr(R)$. The following assertions hold.
\begin{enumerate}
    \item For all $\Gamma$-graded right $R$-modules $M$, $$\soc_{\rm gr}^\infty(M)=\{x\in M\colon xJ^\infty=0\}.$$ 
    \item The following statements are equivalent:
    \begin{enumerate}
        \item $\soc_{\rm gr}^\infty(R_R)=R$.
        \item $J$ is right objectwise nilpotent.
        \item $\soc_{\rm gr}^\infty(M)=M$ for all $\Gamma$-graded right $R$-modules $M$.
    \end{enumerate}
     \end{enumerate}
\end{proposition}
\begin{proof}
 (1)  $R$ is gr-semilocal because it is right $\Gamma_0$-artinian \cite[Proposition~6.2]{chainconditions_Jacobsonradical}.   By \cite[Proposition~6.5]{chainconditions_Jacobsonradical},
 \begin{equation} \label{eq:socn Jn}
    \soc_{\rm gr}^n(M)=\{x\in M\colon xJ^n=0\} 
 \end{equation} for all $\Gamma$-graded right $R$-modules $M$ and $n\in\mathbb{N}$. Thus, if $x\in \soc_{\rm gr}^\infty(M)$, then $xJ^\infty =0$. On the other hand, let $M$ be a $\Gamma$-graded right $R$-module. Suppose that $y$ is a homogeneous element of degree $\gamma\in\Gamma$ that belongs to $\{x\in M\colon xJ^{\infty}=0\}$.     For $e=d(\gamma)$, there exists a positive integer $n_e$ such that $J^{n_e}(e)=J^{\infty}(e)$  by \Cref{prop: J infinity}(1). Hence, $yJ^{\infty}=0$ is equivalent to $yJ^{n_e}=0$. Therefore, $y\in \soc_{\rm gr}^{n_e}(M)$ by \eqref{eq:socn Jn}.

 (2) (a)$\implies$(b): If (a) holds, then $J^\infty=RJ^\infty=\{0\}$ by (1), and it follows from \Cref{prop: J infinity}(3) that $J$ is right objectwise nilpotent.

 (b)$\implies$(c): If (b) holds, then $J^\infty=\{0\}$ by \Cref{prop: J infinity}(3), and it follows from (1) that $\soc_{\rm gr}^\infty(M)=M$.

 (c)$\implies$(a) is clear.
\end{proof} 


\section{Main results}
\label{sec:main}

In the ungraded context, a ring $R$ is semiprimary if it is semilocal and $\rad(R)$ is nilpotent. An important class of examples of semiprimary rings are one-sided artinian rings.
The Hopkins--Levitzki Theorem asserts that, over a semiprimary ring,
a module is noetherian if and only if it is artinian, see for example \cite[Theorem~4.12]{Lam1}. In particular, a one-sided artinian ring is noetherian on the same side.
In the group graded setting, 
graded versions of these results are valid, see \cite[Corollary~2.9.7]{NastasescuVanOystaeyen}, its proof and then proceeed as in the ungraded case \cite[Theorem~4.12]{Lam1}. In the groupoid graded case, one-sided $\Gamma_0$-artinianity does not imply $\Gamma_0$-noetherianity on the same side, see \cite[Corollary~3.59]{chainconditions_Jacobsonradical}. Thus, we must first realize which is the nilpotency-type condition  
on the graded radical that allows to prove the most general version of the Hopkins--Levitzki theorem. The objectwise nilpotence condition turns out to be the right one, but, unlike nilpotence, this condition is not left-right symmetric, see \Cref{prop: contra-exemplos J fort nilp}(4). Furthermore, one-sided $\Gamma_0$-artinian rings are gr-semilocal by \Cref{prop: R gr-art => R/rad(R) gr-ss}, though their graded Jacobson radical need not be left or right objectwise nilpotent, see \Cref{prop: contra-exemplos J fort nilp}(5).

The natural first step toward identifying the correct nilpotency-type condition is to determine the condition under which a one-sided $\Gamma_0$-artinian ring becomes one-sided $\Gamma_0$-noetherian.

\begin{theorem}[Weak Hopkins--Levitzki Theorem]
\label{teo: art => (noet <=> fort nilp)}
    The following assertions are equivalent for a $\Gamma$-graded ring $R$:
    \begin{enumerate}
        \item $R$ is a right  $\Gamma_0$-artinian and a right  $\Gamma_0$-noetherian ring.
        \item 
        $R$ is a right  $\Gamma_0$-artinian ring and $\radgr(R)$ is $\Gamma_0$-finitely generated as a right  ideal.
        \item  $R$ is a right  $\Gamma_0$-artinian ring and $\radgr(R)^k$ is $\Gamma_0$-finitely generated as a right  ideal for some positive integer $k$.        
        \item $R$ is a right  $\Gamma_0$-artinian ring and $\radgr(R)$ is right objectwise nilpotent.
        \item  $R$ is a right  $\Gamma_0$-artinian ring and $\radgr(R)^\infty=\{0\}$.        
        \item $R$ is a right  $\Gamma_0$-artinian ring and $\soc_{\rm gr}^\infty(R_R)=R$. 
        \item $R$ is a gr-semilocal right  $\Gamma_0$-noetherian ring and $\radgr(R)$ is right objectwise nilpotent.
    \end{enumerate}
\end{theorem}

\begin{proof}
    By \Cref{prop: R gr-art => R/rad(R) gr-ss}, $R$ is a gr-semilocal ring in all the items.
    Therefore,  the equivalences $(1)\Leftrightarrow(4)\Leftrightarrow(7)$ follow from \Cref{prop: comp finito = art + noet}(2) and \Cref{prop: G0 comp finito <=> MJn=0}, since $R(e)J^n=1_eRJ^n=1_eJ^n=J^n(e)$ for each $e\in\Gamma_0$ and $n\in\mathbb{N}$, where $J=\radgr(R)$.

    The equivalence $(4)\Leftrightarrow(5)$ was proved in \Cref{prop: J infinity}(3).

The equivalence $(4)\Leftrightarrow(6)$ was proved in \Cref{prop: soc infty}(2).
    
    The implication $(1)\Rightarrow(2)$ follows from \cite[Proposition~3.21(1)]{chainconditions_Jacobsonradical}.

    The implication $(2)\Rightarrow(3)$ is clear.

    Finally, $(3)\Rightarrow(4)$ was proved in \Cref{prop: J infinity}(4).    
\end{proof}

We now present an interesting consequence of \Cref{teo: art => (noet <=> fort nilp)}.

\begin{corollary}
\label{coro: R art => R/J^infty noet}
    Let $R$ be a $\Gamma$-graded ring, and set $J:=\radgr(R)$.
    If $R$ is a right $\Gamma_0$-artinian ring, then $R/J^\infty$ is a right $\Gamma_0$-noetherian ring.
\end{corollary}

\begin{proof}
     By \cite[Corollary 3.23]{chainconditions_Jacobsonradical}, $R/J^\infty$ is a right $\Gamma_0$-artinian ring. By \Cref{prop: J infinity}(5),  $\radgr(R/J^\infty)$ is right objectwise nilpotent. Therefore, $R/J^\infty$ is a right $\Gamma_0$-noetherian ring by \Cref{teo: art => (noet <=> fort nilp)}.
\end{proof}

We have shown in \Cref{teo: art => (noet <=> fort nilp)} that if $R$ right $\Gamma_0$-artinian and right objectwise nilpotent, then $R$ is  right $\Gamma_0$-noetherian. In \Cref{ex:weak_HopLev_but_not_strong_HopLev}, we show that being  right $\Gamma_0$-artinian together with right objectwise nilpotency is not sufficient to guarantee the equivalence of the graded artinian and the graded noetherian properties for graded modules.

Given a gr-semilocal ring whose graded Jacobson radical is one-sided objectwise nilpotent, 
we want to study when gr-artinianity implies gr-noetherianity and viceversa. 
We now present some concepts and a result that are of independent interest and somewhat more general than those required to establish our main results \Cref{teo: left str. G_0-nilp. ==> right noeth.} and \Cref{teo: hopkins-levitski}.

\begin{definition}\label{def:d_finitely_generated_module}
    Let $R$ be a $\Gamma$-graded ring and $M$ a $\Gamma$-graded right $R$-module. 
    \begin{enumerate}
        \item We say that  $M$ is \emph{$d$-finitely generated} if there exist $e_1,\dots,e_n\in\Gamma_0$ such that $M=M1_{e_1}R+\cdots+M1_{e_n}R$.
        
        \item We say that $M$ is \emph{$d$-finitely gr-cogenerated} if, for each family  $\{M_i:i\in I\}$ of graded submodules of $M$ such that, for each $e\in\Gamma_0$,  $\bigcap_{i\in I_e}M_{i}1_e=0$ for some finite subset $I_e\subseteq I$, there exist  $i_1,\dots, i_n\in I$ such that $M_{i_1}\cap\cdots \cap M_{i_n}=0$.
        
    \end{enumerate}
\end{definition}

 Recall from \cite[Definitions 3.1(6) and 3.19(6)]{chainconditions_Jacobsonradical} that a graded module $M$ is called \textit{finitely gr-cogenerated} if, for every family
$\{M_i:i\in I\}$ of graded submodules of $M$ such that
$\bigcap_{i\in I}M_i=0$, there exist $i_1,\dots,i_n\in I$ such that $\bigcap_{t=1}^n M_{i_t}=0$, and $M$ is said to be \textit{$\Gamma_0$-finitely gr-cogenerated} if $M(e)$ is finitely gr-cogenerated for all $e\in\Gamma_0$.

\begin{remark}
\label{rem: fg => d-fg}

    (1) If $\Gamma'_0(R)$ is finite (in particular, if $\Gamma$ is a group), then every $\Gamma$-graded right $R$-module is $d$-finitely generated and $d$-finitely gr-cogenerated.

    (2) Every finitely (gr-co)generated module is $d$-finitely (gr-co)generated.

    (3) Being a $d$-finitely gr-cogenerated module can be regarded as a dual concept of being a $d$-finitely generated module using the following fact: $M$ is $d$-finitely generated if and only if, for each family $\{M_i:i\in I\}$ of graded submodules of $M$ such that, for each $e\in\Gamma_0$,  $\sum_{i\in I_e}M_{i}1_e=M1_e$ for some finite subset $I_e\subseteq I$, there exist $i_1,\dots, i_n\in I$ such that $M_{i_1}+\cdots + M_{i_n}=M$.
     In fact, the ``if part'' follows by considering the family $\{M1_eR:e\in\Gamma_0\}$. Conversely, suppose that  $M=M1_{e_1}R+\cdots+M1_{e_n}R$ for some $e_1,\dots,e_n\in\Gamma_0$, and $\{M_i:i\in I\}$ is a family of graded submodules of $M$ such that, for each $e\in\Gamma_0$, $\sum_{i\in I_e}M_{i}1_e=M1_e$ for some finite subset $I_e\subseteq I$. Then $\sum_{i\in I'}M_i=M$, where $I'=I_{e_1}\cup\cdots\cup I_{e_n}$.
    \qed
\end{remark}

\begin{lemma}
\label{lem: M d-fg => MJ^n=0}
    Let $R$ be a $\Gamma$-graded ring and $M$ be a $\Gamma$-graded right $R$-module. The following assertions hold:
    \begin{enumerate}
        \item If $\radgr(R)$ is right objectwise nilpotent and $M$ is $d$-finitely generated, then there exists $n\in\mathbb{N}$ such that $MJ^n=0$.
        \item If $\radgr(R)$ is left objectwise nilpotent and $M$ is $d$-finitely gr-cogenerated, then there exists $n\in\mathbb{N}$ such that $MJ^n=0$.
    \end{enumerate}
\end{lemma}

\begin{proof}
Let $J:=\radgr(R)$.

   (1) Suppose that $M=M1_{e_1}R+\cdots+M1_{e_k}R$ for some $e_1,\dots,e_n\in\Gamma_0$.
Then, taking $n_1, n_2, \dotsc, n_k\in\mathbb{N}$ such that 
$J^{n_j}(e_j)=0$ for each $j=1,\dotsc, k$, and $n:=\max\{n_j:1\leq j\leq k\}$, we get
\[MJ^n=\sum_{j=1}^kM1_{e_j}RJ^n=\sum_{j=1}^k M1_{e_j}J^{n}=\sum_{j=1}^kMJ^{n}(e_j)\subseteq \sum_{j=1}^kMJ^{n_j}(e_j)=0.\]

(2) Consider the family $\{MJ^n:n\in\mathbb{N}\}$ of graded submodules of $M$.
For each $e\in\Gamma_0$, there exists $n_e\in\mathbb{N}$ such that $(e)J^{n_e}=0$, and thus $MJ^{n_e}1_e=M\cdot (e)J^{n_e}=0$.
Since $M$ is $d$-finitely gr-cogenerated, we have $0=MJ\cap MJ^2\cap\cdots\cap MJ^n=MJ^n$ for some  $n\in\mathbb{N}$. 
\end{proof}

\begin{theorem}[Asymmetric Hopkins--Levitzki Theorem]
\label{teo: left str. G_0-nilp. ==> right noeth.}
Let $R$ be a $\Gamma$-graded ring  and $M$ be a $\Gamma$-graded right  $R$-module. 
The following statements hold:
\begin{enumerate}
\item  If $R$ is gr-semilocal and $\radgr(R)$ is right objectwise nilpotent, then:
\begin{enumerate}
        \item $M$ is gr-noetherian $\Leftrightarrow$ $M$ is gr-artinian and finitely generated.
        \item $M$ is $\Gamma_0$-noetherian $\Leftrightarrow$ $M$ is $\Gamma_0$-artinian and $\Gamma_0$-finitely generated.
        \item $M$ is strongly $\Gamma_0$-noetherian $\Leftrightarrow$ $M$ is strongly $\Gamma_0$-artinian and $\Gamma_0$-finitely generated.
\end{enumerate}

\item  If $R$ is gr-semilocal and $\radgr(R)$ is left objectwise nilpotent, then:
\begin{enumerate}
    \item $M$ is gr-artinian $\Leftrightarrow$ $M$ is gr-noetherian and finitely gr-cogenerated.
    \item $M$ is $\Gamma_0$-artinian $\Leftrightarrow$ $M$ is $\Gamma_0$-noetherian and $\Gamma_0$-finitely gr-cogenerated. 
    \item $M$ is strongly $\Gamma_0$-artinian $\Leftrightarrow$ $M$ is strongly $\Gamma_0$-noetherian and $\Gamma_0$-finitely gr-cogenerated.
\end{enumerate}

\end{enumerate}
\end{theorem}

\begin{proof}
Set $J:=\radgr(R)$.

(1)(a) By \cite[Proposition~3.3(1)]{chainconditions_Jacobsonradical}, both  assertions imply that $M$ is finitely generated. Hence, by \Cref{rem: fg => d-fg}(2),  $M$ is  $d$-finitely generated.  By \Cref{lem: M d-fg => MJ^n=0}(1), we have $MJ^n=0$ for some $n\in\mathbb{N}$.  The result now follows from \Cref{prop: comp finito <=> MJn=0}.

(b) Apply (a) to $M(e)$ for each $e\in\Gamma_0$.

(c) If $M$ is strongly $\Gamma_0$-artinian and $M(e)$ is finitely generated for every $e\in\Gamma_0$, then it follows from \Cref{lem:implications_chain_conditions}(1) and (b) that $M$ is $\Gamma_0$-noetherian. Now it follows that $M$ is strongly $\Gamma_0$-noetherian by \Cref{prop: comp G-finito => fort art = fort noet}(1).
The converse follow similarly, by using \Cref{lem:implications_chain_conditions}(2), item (b), and \Cref{prop: comp G-finito => fort art = fort noet}(2).

(2) This is analogous to (1) using \cite[Proposition~3.3(2)]{chainconditions_Jacobsonradical} and \Cref{lem: M d-fg => MJ^n=0}(2).
\end{proof} 

\begin{remark}
    (1) Let $R$ and $M$ be as in \Cref{teo: left str. G_0-nilp. ==> right noeth.}(1). Then $M$ is gr-artinian and finitely generated if and only if $M$ is gr-artinian and $d$-finitely generated. Indeed, if $M$ is gr-artinian and $d$-finitely generated, then $MJ^n=0$ for some positive integer $n$ by \Cref{lem: M d-fg => MJ^n=0}(1). Hence, $M$ is gr-noetherian by \Cref{prop: comp finito <=> MJn=0}. Thus, $M$ is finitely generated by \cite[Proposition~3.3(1)]{chainconditions_Jacobsonradical}. The other implication follows easily from \Cref{rem: fg => d-fg}(2).

    (2) Let $R$ and $M$ be as in \Cref{teo: left str. G_0-nilp. ==> right noeth.}(2). Then  $M$ is gr-noetherian and finitely gr-cogenerated if and only if $M$ is gr-noetherian and $d$-finitely gr-cogenerated. This statement follows by a similar argument to (1). \qed
\end{remark}

\begin{definition}
    Let $R$ be a $\Gamma$-graded ring. We say that $R$ is gr-semiprimary if $R/\radgr(R)$ is a gr-semisimple ring and $\radgr(R)$ is objectwise nilpotent. 
\end{definition}

Now we are ready to state our groupoid graded version of the Hopkins--Levitzki Theorem \cite[(4.15)]{Lam1}.

\begin{theorem}[Strong Hopkins--Levitzki Theorem]
\label{teo: hopkins-levitski}
    Let $R$ be a $\Gamma$-graded gr-semiprimary ring. 
    The following assertions hold true for a $\Gamma$-graded right (or left) $R$-module $M$.
    \begin{enumerate}
        \item $M$ is gr-noetherian $\Leftrightarrow$ $M$ is gr-artinian $\Leftrightarrow$ $M$ has finite gr-length.
        \item $M$ is $\Gamma_0$-noetherian $\Leftrightarrow$ $M$ is $\Gamma_0$-artinian $\Leftrightarrow$ $M$ has $\Gamma_0$-finite gr-length.
        \item $M$ is strongly $\Gamma_0$-noetherian $\Leftrightarrow$ $M$ is strongly $\Gamma_0$-artinian $\Leftrightarrow$ $M$ has strongly $\Gamma_0$-finite gr-length.
    \end{enumerate}
\end{theorem}

\begin{proof}
(1) 
By \Cref{teo: left str. G_0-nilp. ==> right noeth.}, $M$ is gr-noetherian if and only if $M$ is gr-artinian.
Now the result follows from \Cref{prop: comp finito = art + noet}(1).

(2) This follows from (1) applied to $M(e)$ for each $e\in\Gamma_0$.

(3) If $M$ is strongly $\Gamma_0$-artinian (resp.\ strongly $\Gamma_0$-noetherian), then $M$ has $\Gamma_0$-finite gr-length by \Cref{lem:implications_chain_conditions} and (2), and it follows from \Cref{prop: comp G-finito => fort art = fort noet}(3) that $M$ is strongly $\Gamma_0$-noetherian (resp.\ strongly $\Gamma_0$-artinian).
The result now follows from \Cref{prop: comp finito = art + noet}(3). 
\end{proof}

Now we give our main example of gr-semiprimary rings.

\begin{theorem}
\label{teo: G_0-art => G_0-noeth}
    Let $R$ be a $\Gamma$-graded ring that is both left and right $\Gamma_0$-artinian. The following assertions hold true.
    \begin{enumerate}
        \item $R$ is gr-semiprimary.
        \item $R$ is both a left and a right $\Gamma_0$-noetherian ring.
    \end{enumerate}
\end{theorem}

\begin{proof}
 By \Cref{prop: R gr-art => R/rad(R) gr-ss}, $R$ is a gr-semilocal ring. We proceed to show that   $J:=\radgr(R)$ is  left and right objectwise nilpotent. 
Since $R$ is left and right $\Gamma_0$-artinian,  it follows from \Cref{prop: J infinity}(1) and its left version that, for each $e\in\Gamma_0$ there exists a positive integer $n_e$ such that
\[
J^\infty(e)=J^{n_e+l}(e) \quad \text{ and }  \quad (e)J^\infty=(e)J^{n_e+l} \quad \textrm{ for all } l\geq 0. 
\]
 Therefore, 
\[
J^\infty=\bigoplus_{e\in\Gamma_0}J^{n_e}(e)=\bigoplus_{e\in\Gamma_0}(e)J^{n_e}.
\]
By \Cref{prop: J infinity}(2) and its left version, 
 $J^\infty=J^\infty\cdot J=J\cdot J^\infty$.  
 However, we have the following sharper equality.
\begin{equation}
\label{eq: J infty = J infty J infty}
    J^\infty\cdot J^\infty=J^\infty.
\end{equation}
Indeed, it is clear that $J^\infty\cdot J^\infty\subseteq J^\infty$, because  $J^\infty$ is a (graded) ideal of $R$. 
On the other hand, for each $e\in\Gamma_0$, we have
\[
{J^\infty}=J\cdot J^\infty=J^2\cdot J^\infty=\dotsb=J^{n_e} \cdot J^\infty,
\]
and thus
 \[J^\infty(e)= 1_eJ^\infty=1_eJ^{n_e}\cdot J^\infty=J^{n_e}(e)\cdot J^\infty\subseteq J^\infty \cdot J^\infty.\]

Suppose, by  way of contradiction, that $J^\infty\neq 0$. 
By \eqref{eq: J infty = J infty J infty}, there exists $e\in\Gamma_0$ such that $J^\infty\cdot (e)J^\infty\neq 0$. 
Let now $T$ be a minimal left ideal contained in $(e)J^\infty$ such that $J^\infty \cdot T\neq 0$. 
Let $a\in \h(T)$ be with $J^\infty a\neq 0$. 
Suppose $a\in T_\alpha$ for some $\alpha\in \Gamma e$. 
Clearly $J^\infty a\subseteq T$. 
Now, again by \eqref{eq: J infty = J infty J infty}
\[
J^\infty\cdot (J^\infty a)=(J^\infty \cdot J^\infty)a=J^\infty a\neq 0.
\]
By the minimality of $T$, we get $J^\infty a=T$. Thus, there exists a nonzero $z\in \h(J^\infty)$ such that $za=a$.  
Hence, $(1_{r(\alpha)}-z)a=0$. 
Notice that $z\in R_{r(\alpha)}$ and $z\in J^\infty \subseteq J$. 
Thus, $1_{r(\alpha)}-z$ is invertible in $R_{r(\alpha)}$ by \Cref{prop:outras_caracs_rad(R)}(1). 
This implies $a=0$, a contradiction.
Therefore, $\bigoplus_{e\in\Gamma_0}J^{n_e}(e)=\bigoplus_{e\in\Gamma_0}(e)J^{n_e}=0$, and statement (1) is proved.

Statement (2) follows from (1) and \Cref{teo: hopkins-levitski}(2).
\end{proof}

We recall from \cite[pp.\ 18--19]{Mit} that  a small preadditive category $\mathcal{C}$ is a \emph{right artinian category} (resp.\ \emph{right noetherian category}) if the functor $\mathcal{C}(-,X)$ is an artinian (resp. noetherian) object in the abelian category $\Fun(\mathcal{C}^{op},\mathcal{A}b)$ of additive contravariant functors for each $X\in\mathcal{C}_0$.

\begin{corollary}
    Let $\mathcal{C}$ be a small preadditive category that is both left and right artinian.
    Then $\mathcal{C}$ is both left and right noetherian.
    Furthermore, if $F:\mathcal{C}\to\mathcal{A}b$ is a (covariant or contravariant) additive funtor, then $F$ is artinian if and only if $F$ is noetherian.
\end{corollary}

\begin{proof}
    By \cite[Theorem 8.1(2)]{groupoid_graded_semisimple}, $\mathcal{C}$ is a right/left artinian (resp. noetherian) category if and only if $R_\mathcal{C}$ is a right/left $\Gamma_0$-artinian (resp. $\Gamma_0$-noetherian) ring, where $\Gamma=\mathcal{C}_0\times\mathcal{C}_0$.
    The first statement now follows from \Cref{teo: G_0-art => G_0-noeth}(2).

    The last part of the statement follows from \Cref{teo: G_0-art => G_0-noeth}(1), \Cref{teo: hopkins-levitski}(1), and the identification between additive functors $\mathcal{C}\to\mathcal{A}b$ and certain $\Gamma$-graded $R_\mathcal{C}$-modules given by \cite[Theorem 8.1(1) and Remark 8.2(1)]{groupoid_graded_semisimple}.
\end{proof}

Since nilpotence of $\radgr(R)$ implies left and right objectwise nilpotence of $\radgr(R)$, we obtain the following immediate
consequence of \Cref{teo: hopkins-levitski}. It can also be shown as a corollary of Propositions \ref{prop: comp finito <=> MJn=0}, \ref{prop: G0 comp finito <=> MJn=0}, and \ref{prop: fort G0 comp finito <=> MJn=0}.

\begin{corollary}
\label{coro: hopkins levitski usual}
    Let $R$ be a gr-semilocal ring such that $\radgr(R)$ is nilpotent. 
    The following assertions hold for a $\Gamma$-graded right (or left) $R$-module $M$:
    \begin{enumerate}
        \item $M$ is gr-noetherian $\Leftrightarrow$ $M$ is gr-artinian $\Leftrightarrow$ $M$ has finite gr-length.
        \item $M$ is $\Gamma_0$-noetherian $\Leftrightarrow$ $M$ is $\Gamma_0$-artinian $\Leftrightarrow$ $M$ has $\Gamma_0$-finite gr-length.
        \item $M$ is strongly $\Gamma_0$-noetherian $\Leftrightarrow$ $M$ is strongly $\Gamma_0$-artinian $\Leftrightarrow$ $M$ has strongly $\Gamma_0$-finite gr-length.\qed
    \end{enumerate}
\end{corollary}

The following result indicates which of the artinianity conditions in \Cref{def:chain_conditions} yield gr-semiprimary rings.

\begin{corollary} 
\label{coro: fort art => fort noet}
Let $R$ be a $\Gamma$-graded ring.
The following assertions hold:
\begin{enumerate}
    \item If $R$ is right strongly $\Gamma_0$-artinian, then  $\radgr(R)$ is nilpotent and $R$ is  right strongly $\Gamma_0$-noetherian.
    \item If $R$ is right gr-artinian, then  $\radgr(R)$ is nilpotent and $R$ is right gr-noetherian.
\end{enumerate}
\end{corollary}

\begin{proof}

(1)  By \Cref{prop:nilpotence of rad(R)}(2), $\radgr(R)$ is nilpotent. By \Cref{lem:implications_chain_conditions}(1), $R$ is a right $\Gamma_0$-artinian ring, and thus it is gr-semilocal by \Cref{prop: R gr-art => R/rad(R) gr-ss}.
It follows from \Cref{coro: hopkins levitski usual}(3) that $R$  is a right strongly $\Gamma_0$-noetherian ring.

(2) By \Cref{prop:nilpotence of rad(R)}(3), $\radgr(R)$ is nilpotent. Since $R$ is gr-semilocal by \Cref{prop: R gr-art => R/rad(R) gr-ss}, it follows that $R$ is a right gr-noetherian ring by \Cref{coro: hopkins levitski usual}(1).
\end{proof}

We close this section with the following result.

\begin{proposition}
\label{prop:R_gr-semiprimary=>Re_semiprimary}
    Let $R$ be a $\Gamma$-graded ring. 
    The following assertions hold:
    \begin{enumerate}
        \item If $R$ is gr-semiprimary, then $R_e$ is semiprimary for every $e \in \Gamma_0$.
        \item If $R$ is gr-semiprimary, then $1_e R 1_e$ is gr-semiprimary for every $e \in \Gamma_0$.
        \item The converse in $(2)$ holds if $\Gamma'_0(R)$ is finite (that is, $R$ is unital).
    \end{enumerate}
\end{proposition}

\begin{proof}
    (1) $R_e$ is semilocal by \cite[Proposition 6.13]{chainconditions_Jacobsonradical}, with $\rad(R_e)^n \subseteq \radgr(R)^n(e)$ for all $n \in \mathbb{N}$ by \Cref{prop:outras_caracs_rad(R)}(2).
    
    (2) Each $1_e R 1_e$ is gr-semilocal by \cite[Proposition 6.13]{chainconditions_Jacobsonradical}, and $\radgr(1_e R 1_e)^n \subseteq \radgr(R)^n(e)$ for all $n \in \mathbb{N}$ by \Cref{prop:outras_caracs_rad(R)}(3).

    (3) Suppose that $\Gamma'_0(R)$ is finite and $1_e R 1_e$ is gr-semiprimary for every $e \in \Gamma_0$.
    Then $R$ is gr-semilocal by \cite[Proposition 6.13]{chainconditions_Jacobsonradical}.
    Moreover, $1_e\radgr(R)1_e=\radgr(1_eR1_e)$ is nilpotent for every $e\in\Gamma_0$, and it follows from \Cref{lem: eqiv nil e G_0-nilpotent}(2) that $\radgr(R)$ is $\Gamma_0$-nilpotent.
    By \Cref{lem: soma de nilpotentes}, $\radgr(R)$ is nilpotent.
\end{proof}

The converse of \Cref{prop:R_gr-semiprimary=>Re_semiprimary}(1) fails even in the group graded case. Indeed, if $k$ is a semiprimary ring, then the polynomial ring $R=k[x]$ endowed with the natural $\mathbb{Z}$-grading is not gr-semiprimary because $x^n\in\radgr(R)^n$ for every $n\geq1$, however $R_0=k$ is semiprimary. 
\Cref{coro:when_triangular_is_semiprimary} will show that the converse of \Cref{prop:R_gr-semiprimary=>Re_semiprimary}(2) need not hold if $\Gamma'_0(R)$ is infinite.


\section{Examples involving upper triangular matrix rings}
\label{sec:UT_I(A)}

Throughout this section, we utilize graded  upper triangular matrix rings to construct counterexamples clarifying the conditions in \Cref{def: nilpotency conditions}. First, we show that these nilpotency-type conditions are mutually distinct. Second, we establish that objectwise nilpotency of the graded Jacobson radical fails to be a symmetric property in one-sided $\Gamma_0$-artinian rings. Finally, we demonstrate that the hypotheses of \Cref{teo: hopkins-levitski} cannot be weakened to ensure the equivalence of the graded artinian and noetherian conditions for graded modules.

 To this end, we now define the specific classes of graded triangular matrix rings needed for our examples.
 \begin{definition}
Let $A$ be a unital ring and let $I$ be a nonempty set. Consider the ring $\M_I(A)$ of $I \times I$ matrices with entries in $A$ where only finitely many entries are nonzero.
		Then $\M_I(A)$ has a natural structure of an $I \times I$-graded ring via $\M_I(A)_{(i,j)} := A E_{ij}$, where $E_{ij}$ denotes the elementary matrix with 1 at the $(i,j)$-entry and 0 elsewhere.
        When $I$ is a partially ordered set, we can consider the $I \times I$-graded subring $\UT_I(A)$ of $\M_I(A)$ consisting of matrices $(a_{ij})_{ij}$ such that $a_{ij} = 0$ whenever $i \not\leq j$.
 \end{definition}

In the following result, we compute the powers of the graded Jacobson radical of $\UT_I(A)$, 
generalizing \cite[Proposition 5.15]{chainconditions_Jacobsonradical}.

\begin{proposition}
\label{prop: powers of radgr UT(A)}
    Let $I$ be a partially ordered set, let $A$ be a unital ring, set $R := \UT_I(A)$, and let $n\geq 1$.
    For each $i,j\in I$ with $i\leq j$, set 
    $$m_{i,j,n}:=\max\{0\leq m\leq n: \text{ there exists a chain } i= k_0<k_1<\cdots<k_m= j \text{ in }I\}$$
Then the $n$th power of the graded Jacobson radical of $R$ is given by
    \[
    \radgr(R)^n = \{(a_{ij})_{ij} \in R :  a_{ij} \in\rad(A)^{n-m_{i,j,n}} \text{ for each $i,j$ with $i\leq j$}\},
    \]
    where $\rad(A)^0:=A$.
\end{proposition}

\begin{proof}
    Let $(a_{ij})_{ij}\in \radgr(R)^n$ and fix $i,j\in I$ with $i\leq j$.
    Then $a_{ij}E_{ij}$ is a finite sum of products of the form
    \begin{equation}
    \label{eq: rad(UT(A))^n}
        (a_1E_{il_1})(a_2E_{l_1l_2})\cdots(a_nE_{l_{n-1}j}),
    \end{equation}
    where $a_1,\dots,a_n\in A$, $i=:l_0\leq l_1\leq l_2\leq\cdots\leq l_{n-1}\leq l_n:= j$ in $I$, and $a_tE_{l_{t-1}l_t}\in\radgr(R)$ for each $1\leq t\leq n$.
    Let $N:=\{1\leq t \leq n:l_{t-1}\neq l_t\}$ and note that $N$ has at most $m_{i,j,n}$ elements.
    If $t\in\{1,\dots,n\}\setminus N$, then $a_t\in\rad(A)$ by \cite[Proposition~5.15]{chainconditions_Jacobsonradical}, and it follows that $a_1a_2\cdots a_n\in\rad(A)^{n-|N|}\subseteq\rad(A)^{n-m_{i,j,n}}$.
    Then, \eqref{eq: rad(UT(A))^n} is an element of $\rad(A)^{n-m_{i,j,n}}E_{ij}$, and therefore $a_{ij}E_{ij}\in \rad(A)^{n-m_{i,j,n}}E_{ij}$.

    Conversely, let $i,j\in I$ with $i\leq j$, set $m:=m_{i,j,n}$, and take $a \in\rad(A)^{n-m}$.
    Write $a=\sum_{t=1}^sa_{1,t}a_{2,t}\cdots a_{n-m,t}$, where $a_{1,t},a_{2,t},\dots, a_{n-m,t}\in\rad(A)$ for each $t=1,\dots,s$.
    Take a chain $i= k_0<k_1<\cdots<k_m= j $ in $I$.
    Then
    \[aE_{ij}=\sum_{t=1}^s(a_{1,t}E_{ii})(a_{2,t}E_{ii})\cdots (a_{n-m,t}E_{ii})E_{ik_1}E_{k_1k_2}\cdots E_{k_{m-1}j}\in\radgr(R)^n\]
    by \cite[Proposition 5.15]{chainconditions_Jacobsonradical}.
\end{proof}

As the first main example of this section, we illustrate the nilpotency-type conditions from \Cref{def: nilpotency conditions} in the context of graded upper triangular matrix rings, showing that they are mutually distinct.

\begin{proposition}
\label{prop: nilpotencias radgr UT(A)}
    Let $I$ be a partially ordered set and $A$ be a unital ring. Consider the groupoid  $\Gamma:=I\times I$ and the $\Gamma$-graded ring $R := \UT_I(A)$. 
    The following assertions hold:
    \begin{enumerate}
        \item $\radgr(R)$ is gr-nil if and only if $\rad(A)$ is nil.
        \item $\radgr(R)$ is $\Gamma_0$-nilpotent if and only if $\rad(A)$ is nilpotent.
        \item $\radgr(R)$ is right objectwise nilpotent if and only if $\rad(A)$ is nilpotent and, for each $i\in I$, there exists $n_i\in\mathbb{N}$ such that $I$ contains no chains of the form $i<k_1<k_2<\cdots<k_{n_i}$.
        \item $\radgr(R)$ is left objectwise nilpotent if and only if $\rad(A)$ is nilpotent and, for each $j\in I$, there exists $n_j\in\mathbb{N}$ such that $I$ contains no chains of the form $k_1<k_2<\cdots<k_{n_j}<j$.
        \item $\radgr(R)$ is nilpotent if and only if $\rad(A)$ is nilpotent and there exists $t\in\mathbb{N}$ such that $I$ contains no chains of the form $k_0<k_1<\cdots<k_{t}$.
    \end{enumerate}
\end{proposition}

\begin{proof}
    Set $J:=\radgr(R)$

    (1) By \Cref{lem: eqiv nil e G_0-nilpotent}(1), $J$ is gr-nil if and only if $J_{(i,i)}$ is nil for each $i\in I$. 
    The result now follows from \cite[Proposition 5.15]{chainconditions_Jacobsonradical}, since $J_{(i,i)}=\rad(A)E_{ii}\cong\rad(A)$ for all $i\in I$.

    (2) Analogous to (1), using \Cref{lem: eqiv nil e G_0-nilpotent}(2).

    (3) Suppose that $J$ is right objectwise nilpotent.
    By (2), $\rad(A)$ is nilpotent.
    Suppose, by way of contradiction, that there exists $i\in I$ such that $I$ contains chains of the form $i<k_1<k_2<\cdots<k_{n}$ for any $n\in\mathbb{N}$.
    However, such a chain implies that $E_{ik_n}=E_{ik_1}E_{k_1k_2}\cdots E_{k_{n-1}k_n}\in (J^n)_{(i,k_n)}$ by \cite[Proposition 5.15]{chainconditions_Jacobsonradical}.
    Then, $e=(i,i)\in\Gamma_0$ is such that $J^n(e)\neq0$ for each $n\in\mathbb{N}$, a contradiction.

    Conversely, suppose that $\rad(A)^p=0$ for some $p\in\mathbb{N}$ and, for each $i\in I$, there exists $n_i\in\mathbb{N}$ such that $I$ contains no chains of the form $i<k_1<k_2<\cdots<k_{n_i}$.
    In the language of \Cref{prop: powers of radgr UT(A)}, we have $m_{i,j,n}< n_i$ for all $j\in I$ and $n\in\mathbb{N}$, and it follows that $(J^{p+n_i})_{(i,j)}=0$ for every $j\in I$.
    Therefore, $J^{p+n_i}(e)=0$, where $e=(i,i)$.

    (4)  Analogous to (3).

    (5) Suppose that $J$ is nilpotent.
    By (2), $\rad(A)$ is nilpotent.
    Suppose, by way of contradiction, that $I$ contains chains of the form $k_0<k_1<\cdots<k_t$ for any $t\in\mathbb{N}$.
    However, such a chain implies that $E_{k_0k_t}=E_{k_0k_1}E_{k_1k_2}\cdots E_{k_{t-1}k_t}\in J^t$ by \cite[Proposition 5.15]{chainconditions_Jacobsonradical}.
    Then, $J^t\neq0$ for each $t\in\mathbb{N}$, a contradiction.

    Conversely, suppose that $\rad(A)^p=0$ for some $p\in\mathbb{N}$ and there exists $t\in\mathbb{N}$ such that $I$ contains no chains of the form $k_0<k_1<\cdots<k_t$.
    In the language of \Cref{prop: powers of radgr UT(A)}, we have $m_{i,j,n}< t$ for all $i,j\in I$ and $n\in\mathbb{N}$, and it follows that $J^{p+t}=0$.
\end{proof}

\Cref{prop: nilpotencias radgr UT(A)} allows us to characterize when a graded upper triangular matrix ring is gr-semiprimary.

\begin{corollary}\label{coro:when_triangular_is_semiprimary}
    Let $A$ be a unital ring, and $I$ be a partially ordered set.
    The following assertions are equivalent:
    \begin{enumerate}
        \item $\UT_I(A)$ is a gr-semiprimary $I\times I$-graded ring.
        \item $A$ and $I$ satisfy the following two conditions:
        \begin{enumerate}
            \item $A$ is a semiprimary ring.
            \item For each $i\in I$, there exists $n_i\in\mathbb{N}$ such that $I$ contains neither chains of the form $i<k_1<k_2<\cdots<k_{n_i}$ nor $k_1<k_2<\cdots<k_{n_i}<i$.
        \end{enumerate}
    \end{enumerate}
\end{corollary}

\begin{proof}
This follows from    \cite[Proposition 6.13]{chainconditions_Jacobsonradical} and \Cref{prop: nilpotencias radgr UT(A)}(3,4).
\end{proof}

We now give some examples of partially ordered sets $I$ satisfying \Cref{coro:when_triangular_is_semiprimary}(2)(b). Notice that under this hypothesis, $\UT_I(A)$ is gr-semiprimary if and only if $A$ is semiprimary.

\begin{example}
    (1) If $I=\bigcup_{\lambda\in\Lambda}I_\lambda$ is a disjoint union of finite partially ordered sets $I_\lambda$, where the elements of distinct sets are not comparable, then $I$ satisfies \Cref{coro:when_triangular_is_semiprimary}(2)(b).

    (2) If $I$ is a partially ordered set and there exists $n\in\mathbb{N}$ such that $I$ does not contain chains of length greater than $n$, then \Cref{coro:when_triangular_is_semiprimary}(2)(b) is satisfied. Moreover, by \Cref{prop: nilpotencias radgr UT(A)}(5), $\radgr(\UT_I(A))$ is nilpotent if and only if $\rad(A)$ is nilpotent.
    Examples of such a poset $I$ include the set of prime ideals of a commutative unital ring with finite Krull dimension, and the set of right ideals of a right artinian unital ring.

    (3) Suppose that $I$ is a \emph{fence} (also called a \emph{zigzag}) poset, that is, the elements of $I$ can be written in the following form
    \[x_1<x_2<\cdots <x_{n_1}>x_{n_1+1}>x_{n_1+2}>\cdots>x_{n_1+n_2}<x_{n_1+n_2+1}>\cdots.\]
    Then $I$ satisfies \Cref{coro:when_triangular_is_semiprimary}(2)(b).
    For a visual example, if $I$ is the fence 
    \[x_1<x_2>x_3<x_4>x_5<x_6>\cdots,\]
    that is $x_{2n-1}<x_{2n}>x_{2n+1}$ for every $n\geq1$, then $\UT_I(A)$ can be viewed as the set of matrices of the form
    \[
    \begin{pmatrix}
a_{11} & a_{12} & 0      & 0      & 0      & 0      & \cdots \\
0      & a_{22} & 0      & 0      & 0      & 0      & \cdots \\
0      & a_{32} & a_{33} & a_{34} & 0      & 0      & \cdots \\
0      & 0      & 0      & a_{44} & 0      & 0      & \cdots \\
0      & 0      & 0      & a_{54} & a_{55} & a_{56} & \cdots \\
0      & 0      & 0      & 0      & 0      & a_{66} & \cdots \\
\vdots & \vdots & \vdots & \vdots & \vdots & \vdots & \ddots
\end{pmatrix}
    \]

    (4) A poset $(I,\leq)$ is called a \emph{tree} if $\{i'\in I:i'<i\}$ is well-ordered by $<$ for every $i\in I$.
    If $I$ is a tree such that $\{i'\in I:i'>i\}$ is finite for every $i\in I$, then is straithforward to see that $I$ satisfies \Cref{coro:when_triangular_is_semiprimary}(2)(b).
    For example, suppose that $I$ is the tree 
    $\{x_0,x_1^{(1)},x_2^{(1)},x_2^{(2)},x_3^{(1)},x_3^{(2)},x_3^{(3)}, \dots\}$ where the partial order is given by
    \[x_0< x_n^{(1)}<x_n^{(2)}<x_n^{(3)}<\cdots <x_n^{(n)}\quad (n\geq1),\]
    then $\UT_I(A)$ can be viewed as the set of matrices of the form
    \[
    \begin{pmatrix}
a_{11} & a_{12} & a_{13} & a_{14} & a_{15} & a_{16} & a_{17} & \cdots \\
0      & a_{22} & 0      & 0      & 0      & 0      & 0      & \cdots \\
0      & 0      & a_{33} & a_{34} & 0      & 0      & 0      & \cdots \\
0      & 0      & 0      & a_{44} & 0      & 0      & 0      & \cdots \\
0      & 0      & 0      & 0      & a_{55} & a_{56} & a_{57} & \cdots \\
0      & 0      & 0      & 0      & 0      & a_{66} & a_{67} & \cdots \\
0      & 0      & 0      & 0      & 0      & 0      & a_{77} & \cdots \\
\vdots & \vdots & \vdots & \vdots & \vdots & \vdots & \vdots & \ddots
\end{pmatrix}
    \]

    (5) Note that if a partially ordered set $I$ satisfies \Cref{coro:when_triangular_is_semiprimary}(2)(b), then the opposite poset $I^{op}$ does too.
    In particular, if $I$ is the set 
    $$\{x_\infty,x_1^{(1)},x_2^{(1)},x_2^{(2)},x_3^{(1)},x_3^{(2)},x_3^{(3)}, \dots\}$$ where the partial order is given by
    \[x_n^{(1)}<x_n^{(2)}<x_n^{(3)}<\cdots <x_n^{(n)}<x_\infty\quad (n\geq1),\]
    then $I$ satisfies \Cref{coro:when_triangular_is_semiprimary}(2)(b) and $\UT_I(A)$ can be viewed as the set of matrices of the form
    \[
    \begin{pmatrix} 
a_{11}   & 0        & 0        & 0        & 0        & 0        & \cdots & a_{1\infty} \\ 
0        & a_{22}   & a_{23}   & 0        & 0        & 0        & \cdots & a_{2\infty} \\ 
0        & 0        & a_{33}   & 0        & 0        & 0        & \cdots & a_{3\infty} \\ 
0        & 0        & 0        & a_{44}   & a_{45}   & a_{46}   & \cdots & a_{4\infty} \\ 
0        & 0        & 0        & 0        & a_{55}   & a_{56}   & \cdots & a_{5\infty} \\ 
0        & 0        & 0        & 0        & 0        & a_{66}   & \cdots & a_{6\infty} \\ 
\vdots   & \vdots   & \vdots   & \vdots   & \vdots   & \vdots   & \ddots & \vdots \\ 
0        & 0        & 0        & 0        & 0        & 0        & \cdots & a_{\infty \infty} 
\end{pmatrix}
    \]

    (6) If $I$ is a totally ordered set, then it is straightforward to check that \Cref{coro:when_triangular_is_semiprimary}(2)(b) is equivalent to $I$ being finite.
    \qed
\end{example}

In the next proposition, we show there exist one-sided $\Gamma_0$-artinian rings whose graded Jacobson radical is neither left nor right objectwise nilpotent. Moreover, it demonstrates that there exist one-sided $\Gamma_0$-artinian rings whose graded Jacobson radical is one-sided objectwise nilpotent, but not objectwise nilpotent on the opposite side.

\begin{proposition}
\label{prop: contra-exemplos J fort nilp}
    Let $I$ be a totally ordered set and $A$ be a unital ring. Consider the groupoid $\Gamma := I \times I$ and the $\Gamma$-graded ring $R := \UT_I(A)$. The following assertions hold:
    \begin{enumerate}
        \item If $A$ is right artinian and, for each $i\in I$, there exists $n_i\in\mathbb{N}$ such that $I$ contains no chains of the form $i<k_1<k_2<\cdots<k_{n_i}$,
        then $R$ is a right $\Gamma_0$-artinian ring and $\radgr(R)$ is right objectwise nilpotent.
        \item If $A$ is left artinian and, for each $j\in I$, there exists $n_j\in\mathbb{N}$ such that $I$ contains no chains of the form $k_1<k_2<\cdots<k_{n_j}<j$,
        then $R$ is a left $\Gamma_0$-artinian ring and $\radgr(R)$ is left objectwise nilpotent.
        
        \item If $A$ is left artinian and $I$ is an infinite well-ordered set, then $R$ is a left $\Gamma_0$-artinian ring and $\radgr(R)$ is not right objectwise nilpotent.
        
        \item If $A$ is left artinian and $I = \mathbb{N}$, then $R$ is a left $\Gamma_0$-artinian ring and $\radgr(R)$ is left objectwise nilpotent  but not right objectwise nilpotent.
        
        \item Suppose that $A$ is left artinian and $I$ is the set of ordinals less than $\lambda$, where $\lambda$ is an ordinal greater than the first infinite ordinal $\omega$. 
        Then $R$ is a left $\Gamma_0$-artinian ring and $\radgr(R)$ is neither left nor right objectwise nilpotent.
    \end{enumerate}
\end{proposition}

\begin{proof}
    (1) By \cite[Proposition 3.58(1)]{chainconditions_Jacobsonradical}, $R$ is a right $\Gamma_0$-artinian ring and, by \Cref{prop: nilpotencias radgr UT(A)}(3), $\radgr(R)$ is right objectwise nilpotent.
    
    (2) By \cite[Proposition 3.58(3)]{chainconditions_Jacobsonradical}, $R$ is a left $\Gamma_0$-artinian ring and, by \Cref{prop: nilpotencias radgr UT(A)}(4), $\radgr(R)$ is left objectwise nilpotent.
    
    (3) By  \cite[Corollary~3.60(2)]{chainconditions_Jacobsonradical},  $R$ is a left $\Gamma_0$-artinian ring.
    By \Cref{prop: nilpotencias radgr UT(A)}(3), $\radgr(R)$ is not right objectwise nilpotent.

    (4) This follows from (2) and (3).

    (5) By (3), it suffices to show that $\radgr(R)$ is not left objectwise nilpotent.
    This follows from \Cref{prop: nilpotencias radgr UT(A)}(4), since for each $n\in\mathbb{N}$ we have a chain $1<2<\cdots<n<\omega$ in $I$.
\end{proof}

We showed in \Cref{teo: art => (noet <=> fort nilp)} that if $R$ right $\Gamma_0$-artinian and right objectwise nilpotent, then $R$ is  right $\Gamma_0$-noetherian. In the next example, we show that this is not sufficient to guarantee the equivalence of the graded artinian and the graded noetherian properties for graded modules.
More precisely, items (a)–(c) in \Cref{teo: left str. G_0-nilp. ==> right noeth.} require the assumption of finite (gr-co)generation.

\begin{example}\label{ex:weak_HopLev_but_not_strong_HopLev}
We present examples of gr-semilocal rings (in fact, one-sided $\Gamma_0$-artinian rings) whose graded Jacobson radical is right (left) objectwise nilpotent, but there exist gr-artinian (gr-noetherian) modules that are not gr-noetherian (gr-artinian).
In particular, these rings are not gr-semiprimary.

(1) Let $I:=\mathbb{Z}_{<0}$ be the set of negative integers,  $\Gamma:=I\times I$, and $D$ be a division ring.
By \Cref{prop: contra-exemplos J fort nilp}(1), the $\Gamma$-graded ring $R:=\UT_I(D)$ is right $\Gamma_0$-artinian and its graded Jacobson radical is right objectwise nilpotent.

Let $M=D^{(I)}$. We think of the elements of $M$ as  infinite row vectors with only a finite number of nonzero entries. 
For each $i\in I$, let $s_i\in M$ be the row vector with $i$-entry equal to $1_D$ and any other different entry equal to zero. 
Then $M$ can be regarded as a right $D$-vector space with basis $\{s_i\colon i\in I\}$. 
Hence every element  $m\in M$ can be uniquely expressed as $m=\sum_{i\in I}s_ia_i$ where $a_i\in D$ and $a_i\neq 0$ for only a finite number of $i\in I$. 
We endow $M$ with a structure of $\Gamma$-graded additive group via $M_{(-1,i)}=s_iD$ for each $i \in I$.
Moreover, the natural right action of $R$ on $M$, makes of $M$ a $\Gamma$-graded right $R$-module.
Under this action, we have
\begin{equation*}
\label{eq: skEij=sj}
    s_kE_{ij}=\begin{cases}s_j, &\text{if } k=i\leq j\\
0, &\text{if } k\neq i\leq j   . 
\end{cases}
\end{equation*}
In particular, we obtain the following strict ascending chain
\begin{equation}
\label{eq: siR}
    s_{-1}R\subsetneq s_{-2}R \subsetneq \cdots \subsetneq s_{-n}R \subsetneq \cdots,
\end{equation}
from which $M$ is not gr-noetherian.
We will prove that $M$ is gr-artinian.
It suffices to show that all the nonzero proper graded submodules of $M$ are in the chain \eqref{eq: siR}.
Let  $N$ be a nonzero proper graded submodule of $M$ and let $t$ be the smallest positive integer such that there exists $\sum_{i=-t}^{-1} s_ia_i\in M\setminus N$ with $a_{-t}\neq0$.
Note that  if $-t\geq k$, then $s_k\notin N$ (otherwise, $s_{-1},s_{-2},\dots, s_{-t}\in N$ by \eqref{eq: siR}).
If $m=\sum_{i=k}^{-1} s_ia_i\in N$ for some $k\in I$ with $a_{k}\neq 0$, 
then $s_{k} = m(a_{k}^{-1}E_{kk})\in N$, and thus $k>-t$.
From this and \eqref{eq: siR}, we obtain $t>1$ and $N\subseteq s_{-t+1}R$.
Since $s_{-t+1}\in N$ by the minimality of $t$, we have $N= s_{-t+1}R$, as desired.

(2) Let $I$ be an infinite well-ordered set,  $\Gamma:=I\times I$, and $D$ be a division ring.
By \Cref{prop: contra-exemplos J fort nilp}(2), the $\Gamma$-graded ring $R:=\UT_I(D)$ is left $\Gamma_0$-artinian and its graded Jacobson radical is left objectwise nilpotent.

On the one hand, $R_R$ is right $\Gamma_0$-noetherian according  to \cite[Corollary~3.60(1)]{chainconditions_Jacobsonradical}.
On the other hand, \cite[Corollary~3.60(4)]{chainconditions_Jacobsonradical} shows that $R_R$ is not right $\Gamma_0$-artinian.
\qed
\end{example}



\section{Right gr-hereditary right \texorpdfstring{$\Gamma_0$}{Gamma0}-artinian rings}
\label{sec:gr-hered_gr-art}

The main objective of this section is to introduce right gr-hereditary right $\Gamma_0$-artinian rings. These form an important class of one-sided $\Gamma_0$-artinian rings that are also $\Gamma_0$-noetherian on the same side, thereby satisfying \Cref{teo: art => (noet <=> fort nilp)}. We also show that such rings need not satisfy the equivalences in the Strong Hopkins--Levitzki Theorem~\ref{teo: hopkins-levitski}. To achieve these goals, we develop the theory of the graded Jacobson radical in connection with graded matrix rings and gr-projective modules. Since much of this underlying theory must be developed from scratch, the exposition is somewhat technical. We begin by recalling the necessary terminology and definitions from \cite{groupoid_graded_semisimple}.

 \begin{definition}
 \label{def: M_I(R)(E)}
	Let $R$ be a $\Gamma$-graded ring and $\M_I(R)$ be the set of $I\times I$ matrices with only a finite number of nonzero entries in $R$. 
    
A sequence of nonempty subsets $\overline{\Sigma}=(\Sigma_i)_{i\in I}\in\mathcal{P}(\Gamma)^I$ is \emph{fully matricial} for $R$ if it satisfies the following three conditions:
\begin{enumerate}
    \renewcommand{\labelenumi}{\theenumi)}
    \item The set $I_e=\{i\in I\colon  d(\sigma)=e \textrm{ for some } \sigma\in\Sigma_i\}$
 is finite for each $e\in\Gamma_0$.
 \item For each $i\in I$, the map $\Sigma_i\rightarrow \Gamma_0$, $\sigma\mapsto d(\sigma)$ is injective.
 \item The correspondence $\sigma\mapsto r(\sigma)$ defines a bijective function $\Sigma_i\rightarrow\Gamma_0'(R)$ for each $i\in I$. 
\end{enumerate}
    
    Suppose $\overline{\Sigma}=(\Sigma_i)_{i\in I}\in\mathcal{P}(\Gamma)^I$ is a fully matricial sequence for $R$. 
    For each $\gamma\in \Gamma$, consider the additive subgroup 
	$$\M_I(R)(\overline{\Sigma})_\gamma:=\left\{(a_{ij})\in \M_I(R)\mid a_{ij}\in R_{\Sigma_i \gamma \Sigma_j^{-1}} \right\}.$$
    Note that if $i,j\in I$ and $\gamma\in\Gamma$, then $\Sigma_i \gamma \Sigma_j^{-1}$ is either empty or $\Sigma_i \gamma \Sigma_j^{-1}=\{\sigma_i \gamma \tau_j^{-1}\}$, where $\sigma_i\in\Sigma_i$ is unique in $\Sigma_i$ with  $d(\sigma_i)=r(\gamma)$ and 
    $\tau_j\in\Sigma_j$ is unique in $\Sigma_j$ with $d(\gamma)=d(\tau_j)$.
    By \cite[Proposition 3.7]{groupoid_graded_semisimple},
	$$\M_I(R)(\overline\Sigma):=\bigoplus_{\gamma\in \Gamma} \M_I(R)(\overline{\Sigma})_\gamma$$
    turns $M_I(R)$ into a $\Gamma$-graded ring. Moreover, for each $e\in\Gamma_0$, the unity of $\M_I(R)(\overline\Sigma)_e$ is the matrix 
    $$\mathbb{I}_e:=\sum_{{i\in I_e}}\sum_{\sigma_i\in\Sigma_ie}E_{ii}^{r(\sigma_i)}\in \M_I(R)(\overline\Sigma)_e\,,$$
    where $I_e:=\{i\in I\colon  d(\sigma)=e \textrm{ for some } \sigma\in~\Sigma_i\}$ and $E_{ii}^{r(\sigma_i)}$ is the matrix whose $(i,i)$-entry is $1_{r(\sigma_i)}$ and all its other entries are zero.

    If $i,j\in I$ and $a\in R_{\sigma_i\gamma\tau_j^{-1}}$ for some $\gamma\in \Gamma$, $\sigma_i\in\Sigma_i$ and $\tau_j\in\Sigma_j$ with $d(\sigma_i)=r(\gamma)$ and $d(\gamma)=d(\tau_j)$, then the matrix whose $(i,j)$-entry is $a$ and all its other entries are zero will be denoted by $aE_{ij}$. Note that $aE_{ij}\in\M_I(R)(\overline{\Sigma})_\gamma$.
 \end{definition}

If $U$ is a graded right (resp.\ left) ideal of the $\Gamma$-graded ring $R$, and $\overline{\Sigma}=(\Sigma_i)_{i\in I}\in\mathcal{P}(\Gamma)^I$ is a fully matricial sequence for $R$, we denote 
$$
\M_I(U)(\overline{\Sigma}) := \{(a_{ij})_{ij} \in \M_I(R)(\overline{\Sigma}) :  \, a_{ij} \in U \text{ for all } i,j \in I\}.
$$
It can be easily verified that $\M_I(U)(\overline{\Sigma})$ is a graded right (resp.\ left) ideal of $\M_I(R)(\overline{\Sigma})$.

In classical ring theory, various structural properties transfer seamlessly from a base ring $R$ to the full matrix ring $\mathrm{M}_n(R)$. Now we extend this framework to the graded setting by examining the transfer of gr-semisimplicity and gr-semilocality from a groupoid graded ring $R$ to the corresponding graded matrix ring $\mathrm{M}_I(R)$.
Our next result characterizes  graded two-sided ideals of a graded matrix ring.

\begin{lemma}
\label{lem: ideais dos aneis de matrizes}
Let $R$ be a $\Gamma$-graded ring and let $\overline{\Sigma} = (\Sigma_i)_{i \in I} \in \mathcal{P}(\Gamma)^I$ be a fully matricial sequence for $R$.  
Then the graded ideals of $\M_I(R)(\overline{\Sigma})$ are precisely those of the form $\M_I(U)(\overline{\Sigma})$ for some graded ideal $U$ of $R$.
\end{lemma}

\begin{proof}
Let $V$ be a graded ideal of $\M_I(R)(\overline{\Sigma})$. Fix $i_0 \in I$, and consider
$$
U := \{a \in R : aE_{i_0 i_0} \in V\}.
$$
It is easy to verify that $U$ is a graded ideal of $R$. We claim that $V = \M_I(U)(\overline{\Sigma})$.  

Let $(a_{ij})_{ij} \in \h(V)$. For each $i, j \in I$, let $\gamma_{ij} := \deg a_{ij}$. Then:
\begin{align*}
    a_{ij}E_{ij} = E^{r(\gamma_{ij})}_{ii} (a_{ij})_{ij} E^{d(\gamma_{ij})}_{jj} \in V &\implies a_{ij}E_{i_0 i_0} = E^{r(\gamma_{ij})}_{i_0 i}(a_{ij}E_{ij})E^{d(\gamma_{ij})}_{j i_0} \in V \\
    &\implies a_{ij} \in U.
\end{align*}
Hence, $V \subseteq \M_I(U)(\overline{\Sigma})$.

Conversely, let $(a_{ij})_{ij} \in \h(\M_I(U)(\overline{\Sigma}))$. For each $i,j \in I$, let $\gamma_{ij} := \deg a_{ij}$. Then
$$
a_{ij} \in U \Rightarrow a_{ij}E_{i_0 i_0} \in V \Rightarrow a_{ij}E_{ij} = E^{r(\gamma_{ij})}_{i i_0}(a_{ij}E_{i_0 i_0})E^{d(\gamma_{ij})}_{i_0 j} \in V.
$$
Therefore, $(a_{ij})_{ij} = \sum_{i,j \in I} a_{ij} E_{ij} \in V$.
\end{proof}

Now we show how the graded Jacobson radical behaves under the graded matrix ring construction. It generalizes \cite[Proposition 5.14]{chainconditions_Jacobsonradical}.

\begin{proposition}
\label{prop: radgr(M_I(R))}
Let $R$ be a $\Gamma$-graded ring, and let $\overline{\Sigma} = (\Sigma_i)_{i \in I} \in \mathcal{P}(\Gamma)^I$ be a fully matricial sequence for $R$. 
Then $\radgr\left( \M_I(R)(\overline{\Sigma}) \right) = \M_I \left( \radgr (R) \right)(\overline{\Sigma})$.
\end{proposition}

\begin{proof}

($\supseteq$) It suffices to prove that if $\gamma\in\Gamma$, $i,j\in I$, and $a\in\h(R)$ are such that $aE_{ij}\in\M_I \left( \radgr (R) \right)(\overline{\Sigma})_\gamma$, then $aE_{ij}\in \radgr\left( \M_I(R)(\overline{\Sigma}) \right)$.
Take the unique $\sigma_i \in \Sigma_i$ and $\tau_j \in \Sigma_j$ such that $d(\sigma_i) = r(\gamma)$ and $d(\tau_j) = d(\gamma)$.
Then $a\in\radgr(R)_{\sigma_i \gamma \tau_j^{-1}}$.
We will show that condition (2) in \cref{teo: a carac de radgr(R)} holds.
Let $(x_{kl})_{kl}\in\M_I \left(R \right)(\overline{\Sigma})_{\gamma^{-1}}$.
Then
\begin{equation}
\label{eq: rad matrizes}
    \mathds{I}_{r(\gamma)}-(aE_{ij})\cdot (x_{kl})_{kl}=\mathds{I}_{r(\gamma)}-ax_{ji}E_{ii}-\sum_{l\in I\setminus \{i\}}ax_{jl}E_{il}.
\end{equation}
The $(i,i)$-entry of this matrix is $1_{r(\sigma_i)}-ax_{ji}$. 
Since $x_{ji}\in R_{\Sigma_j\gamma^{-1}\Sigma_i^{-1}}=R_{\tau_j\gamma^{-1}\sigma_i^{-1}}$, it follows from \cref{teo: a carac de radgr(R)}(3) that $1_{r(\sigma_i)}-ax_{ji}\in\U(R_{r(\sigma_i)})$. 
Let $b$ be the inverse of $1_{r(\sigma_i)}-ax_{ji}$ in $R_{r(\sigma_i)}$, that is, $bax_{ji}=b-1_{r(\sigma_i)}=ax_{ji}b$.
Then the matrix $p=\mathds{I}_{r(\gamma)}+(b-1_{r(\sigma_i)})E_{ii}$ is the inverse of $\mathds{I}_{r(\gamma)}-ax_{ji}E_{ii}$ in $\M_I \left(R \right)(\overline{\Sigma})_{r(\gamma)}$.
Therefore, from \eqref{eq: rad matrizes} we get
\begin{align*}
    \left(\mathds{I}_{r(\gamma)}-(aE_{ij})\cdot (x_{kl})_{kl}\right)\cdot p&=\mathds{I}_{r(\gamma)}-\sum_{l\in I\setminus \{i\}}\left(ax_{jl}E_{il}\right)\cdot \left(\mathds{I}_{r(\gamma)}+(b-1_{r(\sigma_i)})E_{ii}\right)\\
    &=\mathds{I}_{r(\gamma)}-\sum_{l\in I\setminus \{i\}}ax_{jl}E_{il},
\end{align*}
which is an invertible element of $\M_I \left(R \right)(\overline{\Sigma})_{r(\gamma)}$, whose inverse is the matrix \linebreak $\mathds{I}_{r(\gamma)}+\sum_{l\in I\setminus \{i\}}ax_{jl}E_{il}$.
Hence, $\mathds{I}_{r(\gamma)}-(aE_{ij})\cdot (x_{kl})_{kl}$ is right invertible in \linebreak $\M_I \left(R \right)(\overline{\Sigma})_{r(\gamma)}$, as desired.

$(\subseteq)$ As we saw in the proof of \Cref{lem: ideais dos aneis de matrizes}, fixing $i\in I$, the graded ideal $J:=\{a\in R: aE_{ii}\in \radgr( \M_I(R)(\overline{\Sigma}) )\}$ satisfies $\radgr\left( \M_I(R)(\overline{\Sigma}) \right) = \M_I \left( J \right)(\overline{\Sigma})$. 
We will prove that $J\subseteq \radgr(R)$. 
Let $\delta\in\Gamma$, $a\in J_\delta$, and $x\in R_{\delta^{-1}}$.
Then $aE_{ii}\in\radgr(\M_I(R)(\overline{\Sigma}))_{\sigma_i^{-1}\delta\tau_i}$ and $xE_{ii}\in\M_I(R)(\overline{\Sigma})_{\tau_i^{-1}\delta^{-1}\sigma_i}$, where $\sigma_i,\tau_i \in \Sigma_i$ are unique such that $r(\sigma_i) = r(\delta)$ and $r(\tau_i) = d(\delta)$, respectively.
By \Cref{teo: a carac de radgr(R)}(3), we have
\[\mathds{I}_{d(\sigma_i)}-(aE_{ii})(xE_{ii})\in\U(\M_I(R)(\overline{\Sigma})_{d(\sigma_i)}),\]
and therefore its $(i,i)$-entry $1_{r(\sigma_i)}-ax$ belongs to $\U(R_{r(\delta)})$.
Thus, $a\in\radgr(R)$ by \Cref{teo: a carac de radgr(R)}.
\end{proof}

Before proceeding, we apply \Cref{prop: radgr(M_I(R))} to show that certain properties pass from the base graded ring to the graded matrix ring, and subsequently present a family of one-sided $\Gamma_0$-artinian rings that are gr-semiprimary.

\begin{corollary}\label{coro: R semilocal ==> M_I(R) semilocal}
    Let $R$ be a $\Gamma$-graded ring and let $\overline{\Sigma} = (\Sigma_i)_{i \in I} \in \mathcal{P}(\Gamma)^I$ be a fully matricial sequence for $R$. 
    The following statements hold:
    \begin{enumerate}
        \item If $R$ is a right $\Gamma_0$-artinian ring, then $\M_I \left( R \right)(\overline{\Sigma})$ is right $\Gamma_0$-artinian.
        \item If $R$ is a right $\Gamma_0$-noetherian ring, then $\M_I \left( R \right)(\overline{\Sigma})$ is right $\Gamma_0$-noetherian.
        \item If $R$ is a gr-semisimple ring, then $\M_I \left( R \right)(\overline{\Sigma})$ is a gr-semisimple ring.
        \item If $R$ is a gr-semilocal ring, then $\M_I \left( R \right)(\overline{\Sigma})$ is a gr-semilocal ring.
        \item If $R$ is a gr-semiprimary ring, then $\M_I \left( R \right)(\overline{\Sigma})$ is a gr-semiprimary ring.
    \end{enumerate}
\end{corollary}

\begin{proof}
(1) Suppose that $R$ is right $\Gamma_0$-artinian. Fix $e\in\Gamma'_0(\M_I(R)(\overline{\Sigma}))$ and $j_0\in I$. Write 
$I_e:=\{i\in I:d(\sigma_i)=e\text{ for some }\sigma_i\in\Sigma_i\}=\{i_1,\dots,i_n\}$
for some $n\in\mathbb{N}$.

Let $V$ be a graded right ideal of $\M_I(R)(\overline{\Sigma})$ such that $V=V(e)$.
Note that if $(a_{ij})_{ij}\in V$ and $a_{ij}\neq0$, then $i\in \{i_1,\dots,i_n\}$.
Consider the set of $j_0$-columns of $V$:
\[V_{j_0}:=\left\{
\begin{pmatrix}
v_1 \\
\vdots \\
v_n
\end{pmatrix}
: \text{ there exists } (a_{ij})_{ij}\in V \text{ such that }
\begin{pmatrix}
v_1 \\
\vdots \\
v_n
\end{pmatrix}
=\begin{pmatrix}
a_{i_1j_0} \\
\vdots \\
a_{i_nj_0}
\end{pmatrix}\right\}.\]
Consider the $\Gamma$-graded ring 
$$R_{j_0}:=\bigoplus_{\gamma\in\Gamma}(R_{j_0})_\gamma,\quad\text{where }(R_{j_0})_\gamma:=R_{\Sigma_{j_0}\gamma\Sigma_{j_0}^{-1}}.$$
By \Cref{def: M_I(R)(E)}(2,3), we have $R_{j_0}=R$ as rings. 
It is then straightforward to see that $V_{j_0}$ is a $\Gamma$-graded right $R_{j_0}$-module, where $(V_{j_0})_\gamma$ consists of the vectors whose $t$-th entry lies in $R_{\Sigma_{i_t}\gamma\Sigma_{j_0}^{-1}}$ for every $t=1,\dots,n$.
In fact, $V_{j_0}$ can be regarded as a graded $R_{j_0}$-submodule of $(R_{j_0}(e))^n$.
We claim that 
\begin{equation}
\label{eq:right_ideals_M_I(R)}
    (a_{ij})_{ij}\in V\iff v_{j'}:=\begin{pmatrix}
a_{i_1j'} \\
\vdots \\
a_{i_nj'}
\end{pmatrix} \in V_{j_0} \text{ for every }j'\in I.
\end{equation}
Indeed, suppose that $(a_{ij})_{ij}\in V_\gamma$ for some $\gamma\in\Gamma$, and let $\sigma\in\Sigma_{j'}$ be such that $d(\gamma)=d(\sigma)$. Then $(a_{ij})_{ij}\cdot E_{j'j_0}^{r(\sigma)}\in V$ and its $(i_t,j_0)$-entry is $a_{i_tj'}$ for every $t=1,\dots,n$. Hence $v_{j'}\in V_{j_0}$.
Conversely, suppose that the right-hand side of \eqref{eq:right_ideals_M_I(R)} holds. Let $J':=\{j'\in I:a_{ij'}\neq0 \text{ for some }i\in I\}$, which is a finite set. Let $j'\in J'$. Since $v_{j'}\in V_{j_0}$, there exists $(b_{ij}^{(j')})_{ij}\in V$ such that $a_{i_tj'}=b_{i_tj_0}$ for every $t=1,\dots,n$. Then $(c_{ij}^{(j')})_{ij}:=\sum_{e\in\Gamma_0}(b_{ij}^{(j')})_{ij}\cdot E_{j_0j'}^{e}\in V$ is such that $a_{i_tj'}=c_{i_tj'}$ for every $t=1,\dots,n$, while every other entry of  $(c_{ij}^{(j')})_{ij}$ is zero. Therefore, $(a_{ij})_{ij}=\sum_{j'\in J'}(c_{ij}^{(j')})_{ij}\in V$, and the claim follows.

Since $R_{j_0}$ and $R$ have the same homogeneous components, it follows that $R_{j_0}$ is a right $\Gamma_0$-artinian ring. Therefore, $(R_{j_0}(e))^n$ is a gr-artinian $R_{j_0}$-module by \cite[Proposition~3.5(1)]{chainconditions_Jacobsonradical}.
Now, let $V_1\supseteq V_2 \supseteq V_3\supseteq\cdots$ be a descending chain of graded submodules of $\big(\M_I(R)(\overline{\Sigma})\big)(e)$. This induces a descending chain 
$$(V_1)_{j_0}\supseteq (V_2)_{j_0}\supseteq (V_3)_{j_0}\supseteq\cdots$$ 
of graded submodules of $(R_{j_0}(e))^n$. This chain stabilizes, since  $(R_{j_0}(e))^n$ is gr-artinian. 
By \eqref{eq:right_ideals_M_I(R)}, the chain $V_1\supseteq V_2 \supseteq V_3\supseteq\cdots$ also stabilizes. It follows that $\big(\M_I(R)(\overline{\Sigma})\big)(e)$ is gr-artinian.

(2) The proof is analogous to that of (1).

(3) It follows from (1), \Cref{prop: radgr(M_I(R))}, and the characterization of gr-semisimple rings given in \cite[Theorem~5.11]{chainconditions_Jacobsonradical}.

(4) By \Cref{prop: radgr(M_I(R))}, $\radgr\left( \M_I(R)(\overline{\Sigma}) \right) = \M_I \left( \radgr(R) \right)(\overline{\Sigma})$.
Therefore,
\[\frac{\M_I(R)(\overline{\Sigma})}{\radgr\left( \M_I(R)(\overline{\Sigma})\right)}=\frac{\M_I(R)(\overline{\Sigma})}{\M_I \left( \radgr(R) \right)(\overline{\Sigma})}\isogr\M_I\left(\frac{R}{\radgr(R)}\right)(\overline{\Sigma}) \]
is a gr-semisimple ring by (3).

(5) The graded ring $\M_I \left( R \right)(\overline{\Sigma})$ is gr-semilocal by (4).  Now set $J:=\radgr(R)$ and $\Tilde{J}:=\radgr\left(\M_I \left( R \right)(\overline{\Sigma})\right)$. By \Cref{prop: radgr(M_I(R))}, $\Tilde{J}=\M_I \left( J \right)(\overline{\Sigma})$. From this, it is straightforward to see that $\Tilde{J}^n=\M_I \left( J^n \right)(\overline{\Sigma})$ for each positive integer $n$. If $e\in \Gamma_0$, the set $I_e=\{i\in I\colon d(\sigma_i)=e \textrm{ for some } \sigma_i\in\Sigma_i\}$ is finite, and there exists a positive integer $n_e$ such that $J^{n_e}(e)=(e)J^{n_e}=0$. 
    For each $n\in\mathbb{N}$ and $\gamma\in\Gamma$, an element of $(\Tilde{J}^n)_\gamma$ is a matrix whose $(i,j)$-entry belongs to $(J^n)_{\Sigma_i\gamma\Sigma_j^{-1}}$.
    For each $e\in\Gamma_0$, let $m_e=\max\{n_{r(\sigma_i)}\colon i\in I_e, \sigma_i\in\Sigma_i, d(\sigma_i)=e\}$. Then $\Tilde{J}^{m_e}(e)=(e)\Tilde{J}^{m_e}=0$.
\end{proof}

\begin{corollary}
\label{coro:R_gr-art=>M_I(R)_G0-semiprimary}
    Let $R$ be a unital $\Gamma$-graded ring and $\overline{\Sigma} = (\Sigma_i)_{i \in I} \in \mathcal{P}(\Gamma)^I$ be a fully matricial sequence for $R$.
    If $R$ is a right gr-artinian ring, then $\M_I(R)(\overline{\Sigma})$ is a right $\Gamma_0$-artinian ring and $\radgr\big(\M_I(R)(\overline{\Sigma})\big)$ is nilpotent.
\end{corollary}

\begin{proof}
    By \Cref{lem:implications_chain_conditions}(1) and \Cref{prop:nilpotence of rad(R)}(3), $R$ is a right $\Gamma_0$-artinian ring, and $\radgr(R)$ is nilpotent.
    The result now follows from \Cref{coro: R semilocal ==> M_I(R) semilocal}(1) and \Cref{prop: radgr(M_I(R))}.
\end{proof}

We now focus on gr-projective modules and how the graded Jacobson radical acts on them.

\begin{definition}
    Let $R$ be a $\Gamma$-graded ring.
We say that a $\Gamma$-graded right $R$-module $P$ is \emph{gr-projective} if, for all $\Gamma$-graded right $R$-modules $M$ and $N$, every surjective $g \in \Homgr(M, N)$ and every $h \in \Homgr(P, N)$, there exists $h' \in \Homgr(P, M)$ such that $g\circ h'=h$ 
    \cite[Proposition 35]{CLP}.
By \cite[Proposition 40 and Definition 28]{CLP}, $P$ is gr-projective if and only if there exist a set $I$ and $(\sigma_i)_{i\in I}\in \Gamma^I$ such that $P$ is  gr-isomorphic to a graded direct summand of $\bigoplus_{i\in I}R(\sigma_i)$.
\end{definition}

We prove the following two results by adapting the ideas from the ungraded case \cite[Propositions 17.10 and 17.14]{AndersonFuller}. 
But first we recall from \cite[Definition~4.1]{chainconditions_Jacobsonradical} that the graded Jacobson radical $\radgr(M)$ of a $\Gamma$-graded right $R$-module $M$ is defined as the intersection of all gr-maximal graded submodules of $M$, provided such submodules exist. If $M$ has no gr-maximal graded submodules, we set $\radgr(M)= M$.

\begin{proposition}
\label{prop: rad(P)}
    Let $R$ be a $\Gamma$-graded ring and $P$ be a gr-projective $\Gamma$-graded right $R$-module.
    Then $\radgr(P)=P\cdot\radgr(R)$.
\end{proposition}

\begin{proof}

    Let $(\sigma_i)_{i\in I}\in\Gamma^I$ and $Q$ be a graded module such that $P\oplus Q\isogr \bigoplus_{i\in I}R(\sigma_i)$.
    By \cite[Corollary 4.7 and Proposition 4.8(1)]{chainconditions_Jacobsonradical}, the graded Jacobson radical preserves direct sums and shifts, so we have
    \[\radgr(P)\oplus\radgr(Q)=\radgr(P\oplus Q)\isogr \radgr\left( \bigoplus_{i\in I}R(\sigma_i) \right)=\bigoplus_{i\in I}\radgr(R)(\sigma_i).\]
    Since $1_{r(\sigma_i)}\in R(\sigma)_{\sigma_i^{-1}}$, the last direct sum is contained in
    \begin{align*}
        \bigoplus_{i\in I}R(\sigma_i)\cdot\radgr(R)=\left(\bigoplus_{i\in I}R(\sigma_i)\right)\radgr(R)&\isogr (P\oplus Q)\radgr(R)\\
        &=(P\cdot\radgr(R))\oplus(Q\cdot\radgr(R)).
    \end{align*}
    The last sum is contained in $\radgr(P)\oplus\radgr(Q)$ by \cite[Corollary 4.6(3)]{chainconditions_Jacobsonradical}.
    Therefore, $\radgr(P)\oplus\radgr(Q)=(P\cdot\radgr(R))\oplus(Q\cdot\radgr(R))$, and it follows that $\radgr(P)=P\cdot\radgr(R)$.
\end{proof}

\begin{proposition}\label{prop: rad(P)=/ P}
    Let $R$ be a $\Gamma$-graded ring and $P$ be a nonzero gr-projective $\Gamma$-graded right $R$-module.
    Then $\radgr(P)\neq P$, that is, $P$ contains a gr-maximal graded submodule.
\end{proposition}

\begin{proof}
    Suppose that $\radgr(P)=P$. We will show that $P=0$.
    Let $Q$ and $F$ be $\Gamma$-graded right $R$-modules such that $P\oplus Q=F$ where $F$ is gr-isomorphic to $\bigoplus_{i\in I}R(\sigma_i)$ for some $(\sigma_i)_{i\in I}\in\Gamma^I$.
    Since $F$ is gr-projective, \Cref{prop: rad(P)} shows that
    \begin{equation}
    \label{eq: P=FJ}
        P=\radgr(P)\subseteq\radgr(P\oplus Q)=\radgr(F)=F\cdot\radgr(R).
    \end{equation}
    For each $i \in I$, let $x_i \in F_{\sigma_i^{-1}}$ be the image of the canonical element in $\bigoplus_{i \in I} R(\sigma_i)$ having $1_{r(\sigma_i)}$ at position $i$ and zero elsewhere.  
    Now note that if $x \in F_\gamma$ for some $\gamma \in \Gamma$, then there exists a unique tuple $(a_i)_{i \in I} \in \bigoplus_{i \in I} R_{\sigma_i \gamma}$ such that $x = \sum_{i \in I} x_i a_i$.
    
    Let $\gamma\in\Gamma$ and $x\in P_\gamma$.

    By \eqref{eq: P=FJ}, there exist indices $i_1, \dots, i_n \in I$, elements $a_l \in \radgr(R)_{\sigma_{i_l} \gamma}$, and elements $a_{kl} \in \radgr(R)_{\sigma_{i_k} \sigma_{i_l}^{-1}}$ for $1 \le k, l \le n$, such that $$  x = \sum_{l=1}^{n} x_{i_l} a_l \quad \text{and} \quad \pi(x_{i_l}) = \sum_{k=1}^{n} x_{i_k} a_{kl} \quad \text{for all } 1 \le l \le n,  $$ where $\pi \colon F \to P$ is the natural projection.
    Hence,
    \[0=x-\pi(x)=\sum_{l=1}^{n}x_{i_l}a_l-\sum_{l=1}^{n}\sum_{k=1}^{n}x_{i_k}a_{kl}a_l=\sum_{k=1}^nx_{i_k}\left(a_k-\sum_{l=1}^{n}a_{kl}a_l\right)\]
    and it follows that 
    \begin{equation}
    \label{eq: P=rad(P)}
        \sum_{l=1}^{n}a_{kl}a_l=a_k\quad\text{ for all } k=1,\dots,n.
    \end{equation}
    For each $1\leq k\leq n$, let 
    \[\Sigma_k:=\begin{cases}
        \{\sigma_{i_k}\}\cup(\Gamma'_0(R)\setminus\{r(\sigma_{i_k})\}),& \text{ if } r(\sigma_{i_k})=d(\sigma_{i_k})\\
        \{\sigma_{i_k},\sigma_{i_k}^{-1}\}\cup(\Gamma'_0(R)\setminus\{r(\sigma_{i_k}),d(\sigma_{i_k})\}),& \text{ if } r(\sigma_{i_k})\neq d(\sigma_{i_k}).
    \end{cases}\]
    Then it is straightforward to see that $\overline{\Sigma}:=(\Sigma_1,\dots,\Sigma_n)$ is a fully matricial sequence for $R$ and $(a_{kl})_{kl}\in \M_n(\radgr(R))(\overline{\Sigma})_{r(\gamma)}$.
    By \Cref{prop: radgr(M_I(R))}, we have $(a_{kl})_{kl}\in \radgr(\M_n(R)(\overline{\Sigma}))_{r(\gamma)}$.
    Since 
    $$(a_{kl})_{kl} \begin{pmatrix}
a_1 \\
\vdots  \\
a_n
\end{pmatrix}
=\begin{pmatrix}
1_{r(\sigma_{i_1})} & \cdots & 0 \\
\vdots & \ddots & \vdots \\
0 & \cdots & 1_{r(\sigma_{i_n})} \\
\end{pmatrix}
\begin{pmatrix}
a_1 \\
\vdots  \\
a_n
\end{pmatrix}
=\mathds{I}_{r(\gamma)}
\begin{pmatrix}
a_1 \\
\vdots  \\
a_n
\end{pmatrix}$$
by \eqref{eq: P=rad(P)}, and $\mathds{I}_{r(\gamma)}-(a_{kl})_{kl}$ is invertible in $\M_n(R)(\overline{\Sigma})_{r(\gamma)}$ by \Cref{teo: a carac de radgr(R)}(3), it follows that $a_1=a_2=\cdots=a_n$, from which $x=0$.
Therefore, $P=0$.
\end{proof}

With these intermediate results in place, we now examine the intended examples.
First, we provide an important class of one-sided $\Gamma_0$-artinian rings that are also $\Gamma_0$-noetherian on the same side, thereby satisfying \Cref{teo: art => (noet <=> fort nilp)}.

\begin{definition}
    Let $R$ be a $\Gamma$-graded ring. We say that $R$ is \emph{right gr-hereditary} if every graded right ideal of $R$ is a gr-projective $R$-module. 
\end{definition}

\begin{corollary}
\label{coro: R G0-art gr-hered => G0-noet}
    Let $R$ be a $\Gamma$-graded ring. If $R$ is  right $\Gamma_0$-artinian and  right gr-hereditary, then $R$ is  right $\Gamma_0$-noetherian.
\end{corollary}

\begin{proof}
Set $J:=\radgr(R)$.
Note that $J^\infty$ is a graded (right) submodule of the gr-projective module $R_R$.  Thus, $J^\infty$ is a right gr-projective module because $R$ is right gr-hereditary. By \Cref{prop: rad(P)} and \Cref{prop: rad(P)=/ P}, $J^\infty\cdot J\neq J^\infty$ if $J^\infty\neq 0$. On the other hand, since $R$ is  right $\Gamma_0$-artinian,  $J^\infty=J^\infty\cdot J$ by \Cref{prop: J infinity}(2). Hence $J^\infty=\{0\}$. By \Cref{coro: R art => R/J^infty noet}, $R\isogr R/J^\infty$ is right $\Gamma_0$-noetherian.
\end{proof}

Using the theory developed for upper triangular matrix rings in the preceding section, we characterize when they are right gr-hereditary and right $\Gamma_0$-artinian.

\begin{proposition}
\label{prop: UT_I(A) gr-hered}
   Let $A$ be a semisimple unital ring, $I$ a totally ordered set, and $\Gamma := I \times I$. Consider the $\Gamma$-graded ring $R := \UT_I(A)$, and set $I_{\ge i} := \{j \in I : i \le j\}$ for each $i \in I$. The following assertions are equivalent:
   
   \begin{enumerate}
       \item $R$ is a right $\Gamma_0$-artinian ring and a right $\Gamma_0$-noetherian ring.
       \item For each $i\in I$, every nonempty subset of $I_{\geq i}$ has a maximum and a minimum.
       \item $R$ is a right $\Gamma_0$-artinian ring and a right gr-hereditary ring.
   \end{enumerate}
\end{proposition}

\begin{proof}
    $(1)\Longleftrightarrow(2)$: It follows from \cite[Proposition~3.58(1,2)]{chainconditions_Jacobsonradical}.

    $(3)\Longrightarrow(1)$: It follows from \Cref{coro: R G0-art gr-hered => G0-noet}.

    $(2)\Longrightarrow(3)$: 
   Let $e = (i_0, i_0) \in \Gamma_0$. We show that every graded submodule of $R(e)$ is gr-projective.

   Let $X\subgr R(e)$.
   By \cite[Lemma~3.57(1)]{chainconditions_Jacobsonradical}, there exists a family $\{U_j: i_0 \le j\}$ of right ideals of $A$ such that $X_{(i_0,j)} = U_j E_{i_0 j}$ and $U_j \subseteq U_{j'}$ if $i_0\leq j \leq j'$.
   Since $A$ is semisimple, we can write
   $$\{U_j: i_0 \le j\}=\{V_1, V_1\oplus V_2, V_1\oplus V_2\oplus V_3,\dots, V_1\oplus\cdots\oplus V_n\}$$ 
   for some positive integer $n$ and right ideals $V_1,\dots,V_n$ of $A$.
   For $1\leq t\leq n$, set 
   $$I_t:=\{j\in I_{\geq i_0}: U_j=V_1\oplus\cdots \oplus V_t\},$$
   $$j_t:=\min I_t,\qquad 
   P_t:=(V_tE_{j_tj_t})R,\qquad\text{and}\qquad \sigma_t:=(j_t,i_0).$$
   Since $V_t$ is a direct summand of $A$, it follows that $P_t$ is a graded direct summand of $E_{j_tj_t}R=R(e_t)$, where $e_t=(j_t,j_t)$, and therefore $P_t$ is gr-projective for all $t=1,\dots,n$.
   Since shifts and direct sums of gr-projective modules are gr-projective, it suffices to construct a gr-isomorphism of $R$-modules
   \[\varphi:X\longrightarrow P_1(\sigma_1)\oplus\cdots\oplus P_n(\sigma_n).\]
   Let $j\in I_{\geq i_0}$, and $aE_{i_0j}\in X_{(i_0,j)}=U_jE_{i_0j}$. 
   Take $1\leq t\leq n$ such that 
   $U_j=V_1\oplus\cdots\oplus V_t$, and write $a=a_1+\cdots+a_t$ with $a_k\in V_k$ for all $k=1,\dots,t$.
   Note that $j_1<j_2<\cdots<j_t\leq j<j_{t+1}<\cdots <j_n$, and define 
   $$\varphi(aE_{i_0j})=(a_1E_{j_1j},\cdots,a_tE_{j_tj},0,\cdots,0)\in (P_1)_{\sigma_1(i_0,j)}\oplus\cdots\oplus(P_n)_{\sigma_n(i_0,j)}.$$ 
   It is straightforward to check that this gives the desired gr-isomorphism.
\end{proof}

We present an example of a graded ring that is right gr-hereditary and right $\Gamma_0$-artinian but fails to satisfy the equivalences in the Strong Hopkins--Levitzki Theorem~\ref{teo: hopkins-levitski}.

\begin{example}
Let $I:=\mathbb{Z}_{<0}$,  $\Gamma:=I\times I$, and $D$ be a division ring.
By \Cref{prop: UT_I(A) gr-hered}, $R:=\UT_I(A)$ is a right $\Gamma_0$-artinian ring and a right gr-hereditary ring.
However, $R$ is not gr-semiprimary as we showed in \Cref{ex:weak_HopLev_but_not_strong_HopLev}(1).\qed
\end{example}


\section{\texorpdfstring{$d$}{d}-finitely generated graded rings}
\label{sec:d-fg}

In this section, we study $d$-finitely generated rings. This class is particularly interesting because several of its fundamental properties can be inferred from a small part of the ring. Furthermore, we establish a connection between these rings and the representation theory of algebras.

\begin{definition}
    Let $R$ be a $\Gamma$-graded ring. We say that $R$ is \emph{$d$-finitely generated} if the regular module $R_R$ is a $d$-finitely generated right $R$-module. In other words, there exists a finite set $\Delta_0\subseteq \Gamma_0$ such that $R=\sum_{e\in\Delta_0}R1_eR$.
\end{definition}

The concept of a $d$-finitely generated ring is symmetric. Specifically, if a graded left $R$-module $M$ is defined as $r$-finitely generated when $M = \sum_{e \in \Delta_0} R 1_e M$ for a finite subset $\Delta_0 \subseteq \Gamma_0$, a graded ring $S$ is $r$-finitely generated precisely when ${}_S S$ is $r$-finitely generated. Under this formulation, being $d$-finitely generated and being $r$-finitely generated are equivalent properties for a graded ring.

   We present some examples of $d$-finitely generated graded rings.

   \begin{examples}
   
       (1) Group graded rings or, more generally, groupoid-graded rings $R$ satisfying the condition that $\Gamma_0'(R) = \{e \in \Gamma_0 \colon 1_e \neq 0\}$ is finite.

(2)  Graded rings that satisfy the ascending chain condition on graded (two-sided) ideals.

       (3) Let $A$ be a unital ring.
    Denote by \mbox{$\proj$-$A$} the category of finitely generated projective right $A$-modules, and let $\mathcal{C}$ be a full subcategory of $\proj$-$A$ with $A_A\in\mathcal{C}_0$. 
    It was shown in the proof of \cite[Proposition 3.31]{chainconditions_Jacobsonradical},  that $R_\mathcal{C}=R_\mathcal{C}1_fR_\mathcal{C}$, where $\Gamma:=\mathcal{C}_0\times\mathcal{C}_0$ and $f:=(A,A)\in\Gamma_0$.

    (4) Let $A$ be a unital ring and $\mathcal{C}$ a small full subcategory of the category of right $A$-modules. Assume that every object in $\mathcal{C}$ has finite length and that there exists a finite set $\mathcal{X} \subseteq \mathcal{C}_0$ of indecomposable $A$-modules such that every object in $\mathcal{C}$ is isomorphic to a finite direct sum of modules from $\mathcal{X}$. 
    The proof of \cite[Proposition 3.37]{chainconditions_Jacobsonradical} establishes that $R_\mathcal{C}=\sum_{f\in\Delta_0}R_\mathcal{C}1_fR_\mathcal{C}$, where $\Delta_0:=\{(X,X):X\in\mathcal{X}\}$.
    \qed
   \end{examples}

One of the main reasons we can infer properties of the entire ring from a small part is the following result.

\begin{lemma}\label{lem:d_fg_iso_direct_sum}
 Let $R$ be a $d$-finitely generated $\Gamma$-graded ring, and write $R = \sum_{e \in \Delta_0} R 1_e R$ for a finite subset $\Delta_0 \subseteq \Gamma_0$. Then for any $f \in \Gamma_0'(R)$, there exist a positive integer $n$, elements $e_1, \dots, e_n \in \Delta_0$, elements $\sigma_i \in f\Gamma e_i$ for $i = 1, \dots, n$, and an injective gr-homomorphism $R(f) \longrightarrow R(\sigma_1) \oplus \cdots \oplus R(\sigma_n).$
\end{lemma}
\begin{proof}
    Let $f\in\Gamma_0$.
    Since $1_f\in\sum_{e\in\Delta_0}R1_eR$, we can write $1_f=\sum_{i=1}^nu_iv_i$ with $u_i\in R_{\sigma_i^{-1}}$, $v_i\in R_{\sigma_i}$, $d(\sigma_i)=f$, and $r(\sigma_i)\in\Delta_0$.
    This allows us to consider the following gr-homomorphisms
    \begin{align*}
        \varphi\colon R(f)&\longrightarrow R(\sigma_1)\oplus\cdots\oplus R(\sigma_n), & \psi\colon R(\sigma_1)\oplus\cdots\oplus R(\sigma_n)&\longrightarrow R(f) \\
        a&\longmapsto (v_1a,\dots,v_na) & (z_1,\dotsc,z_n)&\longmapsto \sum_{i=1}^nu_iz_i.
    \end{align*}
The result follows from the fact that $\psi \varphi$ is the identity map on $R(f)$.
\end{proof}

The next results establish the transfer of structural properties from a part of the graded ring $R$ to the whole ring.

\begin{proposition}
\label{prop: R=R1_eR => G0-art so depende dos finitos e}
    Let $R$ be a $d$-finitely generated $\Gamma$-graded ring, and write $R = \sum_{e \in \Delta_0} R 1_e R$ for a finite subset $\Delta_0 \subseteq \Gamma_0$. Set $R_\Delta = \bigoplus_{e, e' \in \Delta_0} 1_e R 1_{e'}$. 
    The following statements are equivalent:
    \begin{enumerate}
        \item $R$ is a right $\Gamma_0$-artinian (resp.\ right $\Gamma_0$-noetherian, gr-semisimple) ring.
        \item $R_\Delta$ is a right gr-artinian (resp.\ right gr-noetherian, gr-semisimple) ring.
        \item $R(e)$ is a gr-artinian (resp.\ gr-noetherian, gr-semisimple) $R$-module for all $e\in\Delta_0$.
    \end{enumerate}
\end{proposition}

\begin{proof}
    $(1)\Rightarrow(2)$: The artinian and noetherian cases follow from \cite[Lemmas 3.28(2) and 3.20]{chainconditions_Jacobsonradical}, while the gr-semisimple case, with $\Delta_0$ not necessarily finite, was proved in \cite[Proposition~5.32(1)]{groupoid_graded_semisimple}.
    
    $(2)\Rightarrow(3)$: 
    Suppose that $R_\Delta$ is a right gr-artinian ring.
    Let $e\in\Delta_0$ and $X_1\supseteq X_2\supseteq \dotsb \supseteq X_n\supseteq \dotsb$ be a descending chain of graded submodules of $R(e)$.
For each $X\subgr R(e)$, denote $X1_{\Delta_0}:=\sum_{e'\in\Delta_0}X1_{e'}$, so that we have the following descending chain of graded right ideals of $R_\Delta$:
\[X_11_{\Delta_0}\supseteq X_21_{\Delta_0}\supseteq \dotsb \supseteq X_n1_{\Delta_0}\supseteq \dotsb.\]
Therefore, there exists $n\in\mathbb{N}$ such that $X_n1_{\Delta_0}=X_{n+l}1_{\Delta_0}$ for every $l\geq1$.
Then 
\[X_n=X_nR\subseteq\sum_{e'\in\Delta_0}X_nR1_{e'}R\subseteq X_n1_{\Delta_0}R=X_{n+l}1_{\Delta_0}R\subseteq X_{n+l}\]
for every $l\geq1$.
Therefore, $R(e)$ is gr-artinian for each $e\in\Delta_0$.

The noetherian case is proved in an analogous way.

    Suppose that $R_\Delta$ is a gr-semisimple ring.
    Let $e\in\Delta_0$ and write $1_e=s_1+\cdots+s_n$, where $s_1,...,s_n\in R_e$ and $s_iR_\Delta$ is a gr-simple $R_\Delta$-module for all $i=1,...,n$. 
    Fix $i=1,...,n$ and $0\neq a\in s_iR$.
    Then $\sum_{e'\in\Delta_0}aR1_e'$ is a nonzero graded $R_\Delta$-submodule of $s_iR_\Delta$, and therefore $s_i\in\sum_{e'\in\Delta_0}aR1_e'\subseteq aR$.
    Therefore, $s_iR$ is gr-simple, and it follows that $R(e)=1_eR=\sum_{i=1}^ns_iR$ is gr-semisimple  for each $e\in\Delta_0$.

    $(3)\Rightarrow(1)$: Let $f\in\Gamma_0'(R)$.
    By \Cref{lem:d_fg_iso_direct_sum}, there exists a gr-injective gr-homomorphism $R(f)\longrightarrow R(\sigma_1)\oplus\cdots\oplus R(\sigma_n)$ for some  $\sigma_i\in \Gamma$ with $d(\sigma_i)=f$ and $r(\sigma_i)\in\Delta_0$ for $i = 1, \dots, n$.
    Note that $R(\sigma_i)=R(r(\sigma_i))(\sigma_i)$ for each $1\leq i\leq n$.
    The result now follows because gr-artinianity, gr-noetherianity, and gr-semisimplicity are preserved under shifts, finite direct sums, and graded submodules, see \cite[Propositions 3.5(1) and 3.4]{chainconditions_Jacobsonradical} and \cite[Corollary~54]{CLP}.
\end{proof}

\begin{corollary}
\label{coro: R=R1_eR => gr-semilocal so depende dos finitos e}
    Let $R$ be a $d$-finitely generated $\Gamma$-graded ring, and write $R = \sum_{e \in \Delta_0} R 1_e R$ for a finite subset $\Delta_0 \subseteq \Gamma_0$. 
    Then $R$ is gr-semilocal if and only if $1_eR1_e$ is gr-semilocal for every $e\in\Delta_0$.
\end{corollary}

\begin{proof}
    The ``only if'' part follows from \cite[Proposition 6.13]{chainconditions_Jacobsonradical}.

    Suppose that $1_eR1_e$ is gr-semilocal for every $e\in\Delta_0$.
    By \cite[Proposition 6.13]{chainconditions_Jacobsonradical}, $R_\Delta=\bigoplus_{e,e'\in\Delta_0}1_eR1_{e'}$ is a gr-semilocal ring.
    Let $J:=\radgr(R)$ and $\overline{R}:=R/J$.
    From \Cref{teo: a carac de radgr(R)}(3), it is easy to see that $\radgr(R_\Delta)=\bigoplus_{e,e'\in\Delta_0}1_eJ1_{e'}=:J_\Delta$.
    Note that $\overline{R}=\sum_{e\in\Delta_0}\overline{R}\overline{1_e}\overline{R}$, where $\overline{1_e}=1_e+J$, and
    \[\overline{R}_\Delta=\bigoplus_{e,e'\in\Delta_0}1_e\overline{R}1_{e'}\isogr\frac{R_\Delta}{J_\Delta}=\frac{R_\Delta}{\radgr(R_\Delta)}\]
    is gr-semisimple. 
    It follows from \Cref{prop: R=R1_eR => G0-art so depende dos finitos e}(3) that $\overline{R}$ is gr-semisimple.
\end{proof}

\begin{lemma}
\label{lem: R=R1_eR => todas nilpotencias coincidem}
    Let $R$ be a $d$-finitely generated $\Gamma$-graded ring, and write $R = \sum_{e \in \Delta_0} R 1_e R$ for a finite subset $\Delta_0 \subseteq \Gamma_0$. 
    Suppose that $U$ is a graded ideal of $R$. Then $U$ is nilpotent $\Leftrightarrow$ $U$ is right (left) objectwise nilpotent $\Leftrightarrow$ $U$ is $\Gamma_0$-nilpotent $\Leftrightarrow$ $1_eU1_e$ is nilpotent for every $e\in\Delta_0$.
\end{lemma}

\begin{proof}
    All the implications $(\Rightarrow)$ hold by \Cref{lem: implications nilpotency conditions} and \Cref{lem: eqiv nil e G_0-nilpotent}(2).
    Consequently, we assume that $1_e U 1_e$ is nilpotent for all $e \in \Delta_0$, and show that $U$ is itself nilpotent. 
    We have that
    \[U=RUR=\sum_{e,f\in\Delta_0}R1_eRUR1_fR =\sum_{e,f\in\Delta_0}R1_eU1_fR=RU_\Delta R,\]
    where $U_\Delta:=\bigoplus_{e,f\in\Delta_0}1_eU1_f$.
    If $U^n=RU_\Delta^nR$ for some $n\in\mathbb{N}$, then
    \begin{align*}
U^{n+1} &= U^n U = RU_\Delta^n R U_\Delta R \\
&= \sum_{e,f\in\Gamma_0} R(U_\Delta^n 1_e) R(1_f U_\Delta) R = RU_\Delta^n R_\Delta U_\Delta R = RU_\Delta^{n+1} R.
\end{align*}
    Therefore 
    \begin{equation}
    \label{eq: Un = R Un R}
        U^n=RU_\Delta^n R \quad\text{ for all } n\in\mathbb{N}.
    \end{equation}
    Note that $1_eU_\Delta1_e=1_eU1_e$ is nilpotent for each $e\in\Delta_0$, that is, $U_\Delta$ is a $\Delta_0$-nilpotent graded ideal of the graded ring $R_\Delta:=\bigoplus_{e,f\in\Delta_0}1_eR1_f$. 
    Since $\Delta_0$ is finite, it follows from \Cref{lem: soma de nilpotentes} that $U_\Delta$ is nilpotent, which implies that $U$ is nilpotent by \eqref{eq: Un = R Un R}.
\end{proof}

\begin{remark}
\label{rem: nilpotency index}
    Let $R$ be a $d$-finitely generated $\Gamma$-graded ring, and write $R = \sum_{e \in \Delta_0} R 1_e R$ for a finite subset $\Delta_0 \subseteq \Gamma_0$. 
    Let $U$ be a graded ideal of $R$ such that $1_eU1_e$ is nilpotent for every $e\in\Delta_0$. By \Cref{lem: R=R1_eR => todas nilpotencias coincidem}, there exists a positive integer $n$ such that $U^n=0$.
    We can compute such an integer explicitly. 
    For each $e\in\Delta_0$, let $n_e$ be a positive integer such that $(1_eU1_e)^{n_e}=0$. 
    We claim that $U_\Delta^n=0$, where $n=1+\sum_{e\in\Delta_0}n_e$.
    Indeed, each element of $U_\Delta^n$ is a sum of products of the form $x_1x_2\cdots x_n$, with $x_i\in U_\Delta(e_i)$ for some $e_i\in\Delta_0$.
    In such a product, there exists $e\in\Delta_0$ such that $n_e+1$ of the elements in $\{x_1,x_2,\dots, x_n\}$ lie in $U_\Delta(e)$.
    By \eqref{eq: 1eU1e nilp => U G0-nilp2}, we have $U_\Delta(e)^{n_e+1}=0$, and it follows that $U_\Delta^n=0$, as desired.
    Then $U^n=0$ by \eqref{eq: Un = R Un R}.\qed
\end{remark}

\begin{proposition}
\label{prop: R=R1_eR => J nilpot}
    Let $R$ be a $d$-finitely generated $\Gamma$-graded ring, and write $R = \sum_{e \in \Delta_0} R 1_e R$ for a finite subset $\Delta_0 \subseteq \Gamma_0$. 
    Then $\radgr(R)$ is nilpotent $\Leftrightarrow$ $\radgr(R)$ is right (left) objectwise nilpotent $\Leftrightarrow$ $\radgr(R)$ is $\Gamma_0$-nilpotent $\Leftrightarrow$ $\radgr(1_eR1_e)$ is nilpotent for every $e\in\Delta_0$.
\end{proposition}

\begin{proof}
    By \Cref{prop:outras_caracs_rad(R)}(3), $\radgr(1_eR1_e)=1_e\radgr(R)1_e$ for every $e\in\Delta_0$.
    Therefore, the result follows from \Cref{lem: R=R1_eR => todas nilpotencias coincidem}.
\end{proof}

\begin{corollary}
\label{coro: R=R1_eR => gr-semiprimary so depende dos finitos e}
    Let $R$ be a $d$-finitely generated $\Gamma$-graded ring, and write $R = \sum_{e \in \Delta_0} R 1_e R$ for a finite subset $\Delta_0 \subseteq \Gamma_0$. 
    Then $R$ is gr-semiprimary $\Leftrightarrow$ $R$ is gr-semilocal and $\radgr(R)$ is nilpotent $\Leftrightarrow$ $1_eR1_e$ is gr-semiprimary for every $e\in\Delta_0$.
\end{corollary}

\begin{proof}
    This follows from \Cref{coro: R=R1_eR => gr-semilocal so depende dos finitos e} and \Cref{prop: R=R1_eR => J nilpot}.
\end{proof}

\begin{corollary}
\label{coro: R=R1_eR => gr-semiprimary}
    Let $R$ be a $d$-finitely generated $\Gamma$-graded ring, and write $R = \sum_{e \in \Delta_0} R 1_e R$ for a finite subset $\Delta_0 \subseteq \Gamma_0$. 
    If $1_eR1_e$ is a one-sided gr-artinian ring for each $e\in\Delta_0$, then $R$ is gr-semiprimary.
\end{corollary}

\begin{proof}
    By \Cref{prop: R gr-art => R/rad(R) gr-ss} and \Cref{prop:nilpotence of rad(R)}(3), $1_eR1_e$ is gr-semiprimary for every $e\in\Delta_0$.
    It follows from \Cref{coro: R=R1_eR => gr-semiprimary so depende dos finitos e} that $R$ gr-semiprimary.
\end{proof}

\begin{corollary}
\label{coro: R=R1_eR G_0-art => rad(R) nilpot}
    Let $R$ be a $d$-finitely generated $\Gamma$-graded ring. If $R$ is right $\Gamma_0$-artinian, then $\radgr(R)$ is nilpotent and $R$ is a right $\Gamma_0$-noetherian ring.
\end{corollary}

\begin{proof}
    By \Cref{prop: R gr-art => R/rad(R) gr-ss}, $R$ is a gr-semilocal ring.
    By \Cref{teo: R G_0-art => rad(R) G_0-nilpotente}, $\radgr(R)$ is $\Gamma_0$-nilpotent, and it follows from \Cref{prop: R=R1_eR => J nilpot} that $\radgr(R)$ is nilpotent.
    Therefore, $R$ is a right $\Gamma_0$-noetherian ring by \Cref{coro: hopkins levitski usual}(2).
\end{proof}

\begin{example}
    Let $A$ be a right (or left) artinian ring with unity.
    Denote by \mbox{$\proj$-$A$} the category of finitely generated projective right $A$-modules, and let $\mathcal{C}$ be a full subcategory of $\proj$-$A$ with $A_A\in\mathcal{C}_0$.
    By \cite[Proposition 3.31]{chainconditions_Jacobsonradical} and \Cref{coro: R=R1_eR G_0-art => rad(R) nilpot}, $R_\mathcal{C}$ is a one-sided $\Gamma_0$-artinian ring  that is gr-semiprimary with nilpotent graded Jacobson radical.\qed
\end{example}

In the following result, further examples of graded rings satisfying the conditions of \Cref{coro: hopkins levitski usual} are constructed from $d$-finitely generated rings.

\begin{proposition}
\label{prop: rad of the category of finite length mod is semiprimary}
    Let $A$ be a unital ring, and $\mathcal{C}$ be a small full subcategory of the category of right $A$-modules. If any of the following conditions hold, then the $\mathcal{C}_0\times\mathcal{C}_0$-graded ring $R_\mathcal{C}$ is gr-semilocal with $\radgr(R_\mathcal{C})$ being nilpotent. 
    \begin{enumerate}
        \item $A$ is a semiprimary ring, $A_A\in\mathcal{C}_0$, and every object of $\mathcal{C}$ is a finitely generated projective $A$-module.
        \item All objects of $\mathcal{C}$ are modules of finite length, and there exists a finite set of indecomposable $A$-modules $\mathcal{X}\subseteq \mathcal{C}_0$ such that each object of $\mathcal{C}$ is isomorphic to a finite direct sum of elements of $\mathcal{X}$.
        \item $A$ is a representation-finite algebra over a field $k$ and $\mathcal{C}$ is a skeleton of the category $\fgmod$-$A$ of finitely generated right $A$-modules.
    \end{enumerate}
\end{proposition}

\begin{proof}
    (1) 
    Since $A$ is semilocal, $R_\mathcal{C}$ is a gr-semilocal ring by \cite[Proposition 6.16(6)]{chainconditions_Jacobsonradical}.
    It was shown in the proof of \cite[Proposition~3.31]{chainconditions_Jacobsonradical} that $R_\mathcal{C}=R_\mathcal{C}1_{(A,A)}R_\mathcal{C}$.
    Since $\radgr(1_{(A,A)}R_\mathcal{C}1_{(A,A)})=\rad(\mathcal{C}(A,A))$ is nilpotent, it follows from \Cref{prop: R=R1_eR => J nilpot} that $\radgr(R_\mathcal{C})$ is nilpotent.
    
    (2) By \cite[Proposition 6.16(1)]{chainconditions_Jacobsonradical}, $R_\mathcal{C}$ is a gr-semilocal ring.
    Following the proof of \cite[Proposition~3.37]{chainconditions_Jacobsonradical}, we obtain $R_\mathcal{C}=\sum_{e\in\Delta_0}R_\mathcal{C}1_eR_\mathcal{C}$, where $\Delta_0 := \{(X,X) \mid X\in\mathcal{X}\}$.
    Moreover, for each $e=(X,X)\in\Delta_0$, the ring $1_eR_\mathcal{C}1_e=\mathcal{C}(X,X)=\End_A(X)$ has a nilpotent Jacobson radical by \cite[Exercise~21.24]{LamExercises}.
    Therefore, $\radgr(R_\mathcal{C})$ is nilpotent by \Cref{prop: R=R1_eR => J nilpot}.

    (3) This follows from (2) because the number of isomorphism classes of indecomposable $A$-modules is finite.
    For another approach, see \cite[Corollary II.1.12 and Lemma VI.1.2]{Coelho}.
\end{proof}

\begin{remark}
\label{rem: nilpotency index J(mod-A)}
    We can exhibit the integers that attest the nilpotency of the graded Jacobson radical in the items of \Cref{prop: rad of the category of finite length mod is semiprimary}.

    (1) Let $A$ be a unital ring, and $\mathcal{C}$ be a small full subcategory of the category of right $A$-modules such that $A_A\in\mathcal{C}_0$ and every object of $\mathcal{C}$ is a finitely generated projective $A$-module. Then $R_\mathcal{C}=R_\mathcal{C}1_{(A,A)}R_\mathcal{C}$, and it follows from \Cref{rem: nilpotency index} that $\radgr(R_\mathcal{C})^{n+1}=0$ if $\rad(A)^n=0$.

    (2) Let $A$ be a unital ring and $\mathcal{C}$ a small full subcategory of right $A$-modules. Assume all objects of $\mathcal{C}$ are of finite length, and fix a finite set $\mathcal{X} \subseteq \mathcal{C}_0$ of indecomposable $A$-modules such that each object of $\mathcal{C}$ decomposes as a finite direct sum of objects in $\mathcal{X}$. 
    We have $R_\mathcal{C}=\sum_{e\in\Delta_0}R_\mathcal{C}1_eR_\mathcal{C}$, where $\Delta_0:=\{(X,X):X\in\mathcal{X}\}$.
    By \cite[Exercise~21.24]{LamExercises}, $\radgr(R_\mathcal{C})_{(X,X)}^{c(X)}=\rad(\mathcal{C}(X,X))^{c(X)}=0$, where $X\in\mathcal{X}$ and $c(X)$ is the length of $X$.
    By \Cref{rem: nilpotency index}, $\radgr(R_\mathcal{C})^{n}=0$, where $n=1+\sum_{X\in\mathcal{X}}c(M)$.\qed
\end{remark}

Let $k$ be a field, $A$ be a finite dimensional $k$-algebra, and $\fgmod$-$A$ be the category of all finitely generated right $A$-modules. 
The definition and basic properties of the radical $\rad_A$ of $\fgmod$-$A$ can be found, for example, in \cite[\S II.1]{Coelho}, where useful characterizations are also given \cite[Theorem~II.1.17]{Coelho}. 
The radical powers $\rad_A^n$ are defined inductively as follows: for each $n\geq2$ and $A$-modules $M,N$, the additive group $\rad_A^n(M,N)$  consist of compositions $g\circ f$ where $f\in\rad_A(M,X)$ and $g\in\rad_A^{n-1}(X,N)$ for some $A$-module $X$. 
The infinite radical is $\rad^\infty:=\bigcap_{n\geq 1}\rad_A^n$.
For more details, see \cite[pp.\ 55 and 89]{Coelho}. 
We relate these definitions with our concept of graded Jacobson radical in the following result.

\begin{lemma}
\label{lem: rad_A = J(mod-A)}
    Let $A$ be a finite-dimensional $k$-algebra over a field $k$, and let $\mathcal{C}$ be a skeleton of the category $\fgmod$-$A$ of finitely generated right $A$-modules.
    Let $M,N$ be right $A$-modules and $\varphi_M:M'\to M$, $\varphi_N:N'\to N$ be isomorphisms for some $M',N'\in\mathcal{C}_0$.
    Then the assignment $g\mapsto \varphi_N^{-1}g\varphi_M$ defines an isomorphism of additive groups $\rad_A^n(M,N)\to (\radgr(R_\mathcal{C})^n)_{(N',M')}$ for each $n\geq 1$.
\end{lemma}

\begin{proof}
    Let $n\geq 1$ and $g\in\rad^n_A(M,N)$. 
    Then there exist right $A$-modules $M_0,\dots,M_n$ with $M_0=M$ and $M_n=N$, and homomorphisms $g_i\in\rad_A(M_{i-1},M_i)$ for each $1\leq i\leq n$ such that $g=g_ng_{n-1}\cdots g_2g_1$.
    For each $i=1,\dots, n-1$, let $\varphi_i:M_i'\to M_i$ an isomorphism for some $M_i'\in\mathcal{C}_0$.
    Then
    \[\varphi_N^{-1}g\varphi_M=\varphi_N^{-1}g_n\varphi_{n-1}\varphi_{n-1}^{-1}g_{n-1}\cdots g_2\varphi_1\varphi_1^{-1}g_1\varphi_M.\]
    Set $\varphi_0 := \varphi_M$, $\varphi_n := \varphi_N$, $M_0' := M'$, and $M_n' := N'$. Since $\rad_A$ is an ideal of the category $\operatorname{mod}\text{-}A$ (see \cite[Lemma~II.1.6]{Coelho}) and $\rad_A(X,Y) = \radgr(R_\mathcal{C})_{(Y,X)}$ for all $X,Y \in \mathcal{C}_0$ (compare \cite[Theorem~II.1.17(d)]{Coelho} with \Cref{teo: a carac de radgr(R)}(3)), we have
\[
\varphi_i^{-1} g_i \varphi_{i-1} \in \rad_A(M_{i-1}', M_i') = \radgr(R_\mathcal{C})_{(M_i', M_{i-1}')} \quad \text{for all } i = 1, \dots, n.
\]
Therefore, $\varphi_N^{-1} g \varphi_M \in (\radgr(R_\mathcal{C}))^n$.
\end{proof}

We conclude the paper by establishing the connection between $d$-finitely generated rings and the representation theory of algebras.

\begin{proposition}
\label{prop: carac A rep-finita}
    Let $k$ be a field, $A$ be a finite dimensional $k$-algebra, and \mbox{$\fgmod$-$A$} be the category of all finitely generated right $A$-modules.
    Let $\mathcal{C}$ be a skeleton of $\fgmod$-$A$, and set $\Gamma:=\mathcal{C}_0\times\mathcal{C}_0$. 
    The following statements are equivalent:
    \begin{enumerate}
        \item $A$ is representation-finite;
        \item $R_\mathcal{C}$ is a right (resp. left) $\Gamma_0$-artinian ring;
        \item $\radgr(R_\mathcal{C})$ is nilpotent;
        \item $\radgr(R_\mathcal{C})$ is right (resp. left) objectwise nilpotent;
        \item $\bigcap_{n\in\mathbb{N}}(\radgr(R_\mathcal{C}))^n=0$.
    \end{enumerate}
\end{proposition}

\begin{proof}
    $(1)\Rightarrow(2)$: This follows from \cite[Proposition~3.37(2)]{chainconditions_Jacobsonradical}.

    $(2)\Rightarrow(4)$: By \Cref{teo: art => (noet <=> fort nilp)}, it suffices we prove that $\radgr(R_\mathcal{C})$ is $\Gamma_0$-finitely generated as a right ideal and as a left ideal. 

    Let $M\in\mathcal{C}_0$, take indecomposable modules $M_1,\dots,M_n\in\mathcal{C}_0$ such that $M\cong M_1\oplus\cdots\oplus M_n$, and let $p_i:M\to M_i$ and $q_i:M_i\to M$ be the corresponding canonical maps.
    
    For each $i=1,\dots,n$, by \cite[Corollary II.3.13]{Coelho}, there exist a right almost split morphism $g_i:X_i\to M_i$ and a left almost split morphism $f_i:M_i\to Y_i$.
    That is, for each $i=1,\dots,n$, $g_i\in\rad_A(X_i,M_i)$, $f_i\in\rad_A(M_i,Y_i)$ and, for all $v\in\rad_A(V,M_i)$ and $u\in\rad_A(M_i,U)$, there exist morphisms $v':V\to X_i$ and $u':Y_i\to U$ such that $v=g_iv'$ and $u=u'f_i$ \cite[Definition II.2.10]{Coelho}.

    Take $X,Y\in\mathcal{C}_0$ such that $X_1\oplus\cdots\oplus X_n\cong X$ and $Y_1\oplus\cdots\oplus Y_n\cong Y$, and let $\pi_i:X\to X_i$ and $\iota_i:Y_i\to Y$ be the corresponding canonical maps. 
    Let 
    $$x:=q_1g_1\pi_1+\cdots+q_ng_n\pi_n\in \rad_A(X,M)=\radgr(R_\mathcal{C})_{(M,X)};$$ 
    $$y:=\iota_1f_1p_1+\cdots+\iota_nf_np_n\in \rad_A(M,Y)=\radgr(R_\mathcal{C})_{(Y,M)}.$$
    Set $J:=\radgr(R_\mathcal{C})$ and $e:=(M,M)\in\Gamma_0$.
    We claim that $J(e)=xR_\mathcal{C}$ and $(e)J=R_\mathcal{C}y$.
    In fact, let $U,V\in\mathcal{C}_0$, $v\in J_{(M,V)}$, and $u\in J_{(U,M)}$.
    For each $i=1,\dots,n$, we have $p_iv\in J_{(M_i,V)}=\rad_A(V,M_i)$ and $uq_i\in J_{(U,M_i)}=\rad_A(M_i,U)$, and it follows that there exist morphisms $v'_i:V\to X_i$ and $u'_i:Y_i\to U$ such that $p_iv=g_iv'_i$ and $uq_i=u'_if_i$.
    Let $v':V\to X$ and $u':Y\to U$ be morphisms such that $\pi_iv'=v_i'$ and $u'\iota_i=u'_i$ for all $i=1,\dots,n$.
    Therefore, we have that $v=\id_Mv=\sum_{i=1}^nq_ip_iv=\sum_{i=1}^nq_ig_iv'_i=\sum_{i=1}^nq_ig_i\pi_iv'=xv'\in xR_{\mathcal{C}}$ and
    $u=u\id_M=\sum_{i=1}^nuq_ip_i=\sum_{i=1}^nu'_if_ip_i=\sum_{i=1}^nu'\iota_if_ip_i=u'y\in R_{\mathcal{C}}y$.

    Hence, $\radgr(R_\mathcal{C})$ is $\Gamma_0$-gr-cyclic as a right ideal and as a left ideal, as desired.

    $(4)\Rightarrow(5)$: Let $J:=\radgr(R_\mathcal{C})$. 
    Suppose that $J$ is right objectwise nilpotent, that is, there exists $(n_e)_{e\in\Gamma_0}\in\mathbb{N}^{\Gamma_0}$ such that $J^{n_e}(e)=0$ for every $e\in\Gamma_0$.
    Then $(\bigcap_{n\in\mathbb{N}}J^n)(e)=\bigcap_{n\in\mathbb{N}}J^n(e)=0$ for every $e\in\Gamma_0$, and it follows that $\bigcap_{n\in\mathbb{N}}J^n=0$.
    The left case is analogous.

    $(5)\Rightarrow(1)$: Suppose that $\bigcap_{n\in\mathbb{N}}(\radgr(R_\mathcal{C}))^n=0$.
    By \Cref{lem: rad_A = J(mod-A)}, we have $\rad_A^\infty=0$, and it follows from \cite[Theorem VI.1.5]{Coelho} that $A$ is representation-finite.

    $(1)\Rightarrow(3)$: Follows from \Cref{prop: rad of the category of finite length mod is semiprimary}(3). 

    $(3)\Rightarrow(5)$: This is immediate.
\end{proof}

\begin{remark}
    Let $k$ be a field, $A$ be a finite dimensional $k$-algebra, $\fgmod$-$A$ be the category of all finitely generated right $A$-modules, $\mathcal{C}$ be a skeleton of $\fgmod$-$A$, and $\Gamma:=\mathcal{C}_0\times\mathcal{C}_0$.

    (1) In the proof of $(2)\Rightarrow(4)$ in \Cref{prop: carac A rep-finita}, we saw that $\radgr(R_\mathcal{C})$ is $\Gamma_0$-gr-cyclic as a right ideal and as a left ideal.
    It follows from \Cref{lema: M fg e J G0-fg => MJ fg} that $\radgr(R_\mathcal{C})^n$ is $\Gamma_0$-gr-cyclic as a right ideal and as a left ideal for every $n\in\mathbb{N}$.
    This is also a consequence of \cite[Lemma 7.10]{AuslanderReitenSmalo1995} and \Cref{lem: rad_A = J(mod-A)}.

    (2) If $A$ is a representation-finite algebra, then it follows from \Cref{rem: nilpotency index J(mod-A)}(2) and \Cref{lem: rad_A = J(mod-A)} that $\rad_A^{c(X_1)+c(X_2)+\cdots+c(X_s)+1}=0$ where $X_1,X_2,\dots,X_s$ are the finite (up to isomorphism) indecomposable finitely generated right $A$-modules.
    In particular, $\rad_A^{sm+1}=0$ if $m$ is a bound on the length of  indecomposable finitely generated right $A$-modules and $s$ is the number of indecomposable finitely generated right $A$-modules up to isomorphism.
    Compare with \cite[Lemma VI.1.2]{Coelho}. \qed
\end{remark}

\bibliographystyle{amsplain}

\end{document}